\documentclass[graybox]{svmult}

\smartqed
\usepackage{mathptmx}       
\usepackage{helvet}         
\usepackage{courier}        
\usepackage{type1cm}        
\usepackage{graphicx}       
\usepackage[ruled,vlined]{algorithm2e}
\usepackage{array,colortbl,todonotes}
\usepackage{amsmath,amsfonts,amssymb,bm,mathtools} 

\numberwithin{figure}{section}
 \usepackage[nodayofweek]{datetime}

\usepackage[nodayofweek]{datetime}

\usepackage{caption}
\usepackage{subcaption}

\newcommand{\bbE}{{\mathbb{E}}}
\newcommand{\bbR}{{\mathbb{R}}}
\newcommand{\calM}{{\mathcal{M}}}

\newcommand{\smsec}{\sum_{i,j=1}^d}

\newcommand{\hLip}{\textbf{(Lip)}{ }}
\newcommand{\hkreg}{\textbf{(Ker-Reg)}{ }}

\newcommand{\hvdiffs}{\textbf{(v-diff-Reg+)}}

\newcommand{\defeq}{\vcentcolon=}

\newcommand{\smfir}{\sum_{j=1}^d}
\newcommand{\smfi}{\sum_{i=1}^r}

\usepackage[colorlinks=true,linkcolor=black,citecolor=black,urlcolor=black]{hyperref}
\usepackage{bookmark}
\makeatletter 
  \providecommand*{\toclevel@author}{999}
  \providecommand*{\toclevel@title}{0}
\makeatother

\begin{document}

\title*{Iterative Multilevel density estimation for McKean-Vlasov SDEs via projections}
\titlerunning{Multilevel Path Simulation for McKean-Vlasov SDEs}
\author{Denis Belomestny \and Lukasz Szpruch \and Shuren Tan}
\institute{
Denis Belomestny
\at Duisburg-Essen University, Duisburg, Germany
\\ \email{denis.belomestny@uni-due.de}
\and
Lukasz Szpruch 
\at University of Edinburgh, School of Mathematics, Edinburgh, EH9 3JZ, UK\\ 
The Alan Turing Institute, London, UK   \\
\email{l.szpruch@ed.ac.uk}
\and Shuren Tan   
\at University of Edinburgh\\ 
         School of Mathematics \\
         Edinburgh, EH9 3JZ, UK,\\ 
         \email{shuren.tan@ed.ac.uk}      
         }
\maketitle

\abstract{In this paper, we present a generic methodology for the efficient numerical approximation of the density function of the McKean-Vlasov SDEs. The weak error analysis for the projected process motivates us to combine the iterative Multilevel Monte Carlo method for McKean-Vlasov SDEs \cite{szpruch2019} with non-interacting kernels and projection estimation of particle densities \cite{belomestny2018projected}. By exploiting smoothness of the coefficients for McKean-Vlasov SDEs, in the best case scenario (i.e $C^{\infty}$ for the coefficients), we obtain the complexity of order $O(\epsilon^{-2}|\log\epsilon|^4)$ for the approximation of expectations and $O(\epsilon^{-2}|\log\epsilon|^5)$ for density estimation.}

\section{Introduction}
Nonlinear Markov processes are stochastic processes whose transition functions
 depends on the state and 
the distribution of the process. They were introduced by 
McKean~\cite{mckean1966class} to give probabilistic interpretation to kinetic PDEs such as Boltzman equations. These powerful framework has been 
studied by a number of authors; we mention here the books of
Kolokoltsov~\cite{kolokoltsov2010nonlinear} and
Sznitman~\cite{sznitman1991topics}. These processes arise naturally in the
study of the limit behaviour of a large number of weakly interacting particle systems and have a wide range of applications, including financial
mathematics, population dynamics, and neuroscience (see, e.g.,
\cite{frank2004stochastic} and the references therein). 
\par
Fix $T>0$ and let $\{W_t\}_{t\in [0,T]}$ be an $r$-dimensional Brownian motion on a filtered probability space $(\Omega, \{ \mathcal{F}_t\}_{t}, \mathcal{F}, \mathbb{P})$. The McKean-Vlasov SDEs with interacting kernels is governed by
\begin{equation} \label{eq:mlmceulerlineartargetinteract}
\begin{cases}
	dX_t =  \int_{\bbR^d}\bar{b}(X_t, y)\mu_t(dy)dt + \int_{\bbR^d}\bar{\sigma}(X_t,  y)\mu_t(dy)dW_t,\\
	 \mu_{t}=\text{Law}(X_t), \quad X_0\in \bbR^d.
\end{cases}
\end{equation} 
Notice that $\{X_t\}_{t\in [0,T]}$ is not a Markov process but $(X_t, \mu_t)_{t\in[0,T]}$ is. It can be shown (under suitable smoothness conditions) that solution to \eqref{eq:mlmceulerlineartargetinteract} has a density function $\mu_{t}(x)$ that satisfies a
nonlinear Fokker-Planck equation of the form
\begin{align*}
\partial_t\mu_{t}(x)  &  =-\sum_{i=1}^d{\partial_{x_i}}\left(
\left(  \int \bar{b}_i(x,y)\mu_{t}(y)\,dy\right)  \,\mu_{t}(x)\right) \nonumber\\
&  +\sum_{i,j=1}^d\frac{1}{2}{\partial^2_{x_i x_j}}
\left( \left[  \left(  \int
\bar{\sigma}(x,y)\mu_{t}(y)\,dy\right)\left(   \int
\bar{\sigma}(x,y)\mu_{t}(y)\,dy\right)^T\right]_{i,j}  \,\mu_{t}(x)  \right), \label{eq:FP}
\end{align*}
which can be seen as an analogue of the linear Fokker-Planck equation in the
SDE case. The theory of the propagation of chaos developed in \cite{sznitman1991topics},
states that \eqref{eq:mlmceulerlineartargetinteract} is a limiting equation of the system of stochastic
interacting particles (samples) with the following dynamics
\begin{equation*}
\label{eq:par} 
\left\{
  \begin{array}{l l}
d Y^{i,N}_t &=   \int_{\bbR^d}\bar{b}(Y^{i,N}_t,y)\mu^{Y,N}_t(dy)  dt +  \int_{\bbR^d}\bar{\sigma}(Y^{i,N}_t,y)\mu^{Y,N}_t(dy)  dW^i_t, 
\\
\mu^{Y,N}_t &:= \frac{1}{N} \sum_{i=1}^{N} \delta_{Y_t^{i,N}}, \quad t\in[0,T],
  \end{array} \right.	
\end{equation*}
where $\{Y^{i,N}_0\}_{i=1,\ldots,N}$ are i.i.d samples with law $\mu_0$ and $\{W^i_t\}_{i=1,\ldots,N}$ are independent Brownian motions. Consider a partition $\{t_k\}$ of $[0,T]$, with $t_k-t_{k-1}=h$ and define $\eta{(t)}\defeq t_k\, \text{if}~t\in[t_k,t_{k+1})$. Authors in \cite{antonelli2002rate,bossy1997stochastic} proposed to use the Euler scheme yields
\begin{equation} \label{eq:eulerpartcle}
\bar{Y}_{t_{k+1}}^{i,N}=\bar{Y}_{t_{k}}^{i,N}+\frac{1}{N}\sum_{j=1}^{N}%
\bar{b}(\bar{Y}_{t_{k}}^{i,N},\bar{Y}_{t_{k}}^{j,N})\,h+\frac{1}{N}\sum_{j=1}%
^{N}\bar{\sigma}(\bar{Y}_{t_{k}}^{i,N},\bar{Y}_{t_{k}}^{j,N})\,(W_{t_{k+1}}^{i}-W_{t_k}^{i})
\end{equation}
for $i=1,\ldots,N.$ Note that implementation of the above algorithm requires 
$O(N^{2})$ numerical operations at each timestep. It has been proved that, under suitable regularity conditions,  these particle systems converge with the weak rate of order $((\sqrt{N})^{-1} + h)$, see \cite{antonelli2002rate,bossy2004optimal,bossy2002rate,bossy1997stochastic} for more details. In some cases, using weaker notation of error in number of particles is $1/N$ as shown in \cite{chassagneux2019weak,szpruch2019antithetic}. Propagation of chaos property suggests using single ensemble of particles to approximate both coefficients of
\eqref{eq:mlmceulerlineartargetinteract} and the quantity of interest (either expected value of the functional or the density function itself). This leads to a complicated nonlinear relation between statistical error and bias. 

Our main motivation is to approximate density function $\mu_t(y)$ of \eqref{eq:mlmceulerlineartargetinteract}. We sketch the key idea here and refer readers to Section \ref{sec:desmlmc} for details. Let $(\varphi_{\mathbf{k}},$ $\mathbf{k}\in \mathbb{N}_0^d)$ be a total orthonormal basis in \thinspace$L_{2}(\mathbb{R}^{d},w)$ for some  weight function $w:\mathbb{R}^d\rightarrow\mathbb{R}$. Provided \(\bar{b}(x,\cdot),\bar{\sigma}(x,\cdot)\in L_2(\bbR^d,w)\) for any \(x\in \bbR^d,\) one can show that smooth density $\mu_{t}(y)$ admits the representation below, for $t>0$,
\begin{equation*}\label{eq:mutdef}
\mu_{t}(y)=\sum_{\mathbf{k}\in \mathbb{N}_0^d}\gamma_{\mathbf{k}}(t)\varphi_{\mathbf{k}}(y)w(y)~\text{with}~\gamma_{\mathbf{k}}(t)\defeq\bbE
\left[  \varphi_{\mathbf{k}}(X_{t})\right]=\int_{\mathbb{R}^{d}}\mu_{t}(y)\varphi_{\mathbf{k}}(y)\,dy.  
\end{equation*}
We follow \cite{belomestny2018projected} and fix some natural number $K > 0$. Consider the projection estimator,
\begin{equation}\label{eq:projestimator}
\mu_t^K(y) =  \sum_{|\mathbf{k}|\leq K}\gamma_{\mathbf{k}}(t)\varphi_{\mathbf{k}}(y)w(y)
\end{equation}
Considering the mean-square-error (MSE) for ${\mu}_{t}^{K}(y)$, one can show
\begin{align}\label{eq:msegeneral}
MSE^K\defeq\sup_{0< t\leq T}\int|{\mu}_{t}^{K}(y)-\mu_{t}(y)|^{2}w^{-1}(y)\,dy=\sum_{|\mathbf{k}|>K}\sup_{0< t\leq T}|\gamma_{\mathbf{k}}(t)|^{2}.
\end{align}
The MSE is controlled entirely by the truncation error due to the projection. Moreover, \eqref{eq:projestimator} inspires us to consider another projection estimator
\begin{align*}
\mu_t^{K,\tilde{X}^K}(y) =  \sum_{|\mathbf{k}|\leq K}\gamma_{\mathbf{k}}^{\tilde{X}^K}(t)\varphi_{\mathbf{k}}(y)w(y),~ \text{with}~\gamma_{\mathbf{k}}^{\tilde{X}^K}(t) = \bbE[\varphi_{\mathbf{k}}(\tilde{X}_{t}^K)],
\end{align*}
where $\tilde{X}_t^K$ is defined as (see full expression in \eqref{eq:defxtminteract_project_K})
\begin{align} \label{eq:mlmceulerlineartargetinteractprojnew}
d\tilde{X}_t^K = \int_{\bbR^d}\bar{b}(\tilde{X}_t^K, y)\mu_t^{K,\tilde{X}^K}(y)dydt + \int_{\bbR^d}\bar{\sigma}(\tilde{X}_t^K,  y)\mu_t^{K,\tilde{X}^K}(y)dydW_t,~\tilde{X}^K_0\sim \mu_0
\end{align}
By decomposition of MSE, one can show $MSE^{K,\tilde{X}}$ is such that\footnote{For convenience, we use the notation $x\lesssim y$ to denote that there exists a constant $c$  such that $x\leq c\ y$.}
\begin{align}
MSE^{K,\tilde{X}^K} &= \sup_{0< t\leq T}\int|{\mu}_{t}^{K,\tilde{X}^K}(y)-\mu_{t}(y)|^{2}w^{-1}(y)\,dy\nonumber\\
&\lesssim\sup_{0< t\leq T}\int|{\mu}_{t}^{K,\tilde{X}^K}(y)-{\mu}_{t}^{K}(y)|^{2}w^{-1}(y)\,dy+\sup_{0< t\leq T}\int|{\mu}_{t}^{K}(y)-{\mu}_{t}(y)|^{2}w^{-1}(y)\,dy\nonumber\\
&=\sup_{0< t\leq T}\sum_{|\mathbf{k}|\leq K}|\gamma_{\mathbf{k}}^{\tilde{X}^K}(t)-\gamma_{\mathbf{k}}(t)|^2+\sum_{|\mathbf{k}|>K}\sup_{0< t\leq T}|\gamma_{\mathbf{k}}(t)|^{2}.\label{eq:MSEintrores}
\end{align}
The weak error analysis (see Lemma \ref{lm:wkmain}) tells us 
\begin{align}\label{eq:weakresintro}
	\sup_{0< t\leq T}\sum_{|\mathbf{k}|\leq K}|\gamma_{\mathbf{k}}^{\tilde{X}^K}(t)-\gamma_{\mathbf{k}}(t)|^2\lesssim\sum_{|\mathbf{k}|>K}\sup_{0< t\leq T}|\gamma_{\mathbf{k}}(t)|^{2}
\end{align}
Combing \eqref{eq:MSEintrores} and \eqref{eq:weakresintro} results in 
\begin{align}\label{eq:MSEXK}
	MSE^{K,\tilde{X}^K}\lesssim\sum_{|\mathbf{k}|>K}\sup_{0< t\leq T}|\gamma_{\mathbf{k}}(t)|^{2}.
\end{align}
Combining \eqref{eq:msegeneral} and \eqref{eq:MSEXK} shows we do not impede the estimation of density $\mu_t(y)$  if considering the dynamics \eqref{eq:mlmceulerlineartargetinteractprojnew} rather than \eqref{eq:mlmceulerlineartargetinteract}. Let us introduce the functions
\begin{align}\label{abd}
\alpha_{\mathbf{k}}(x)   :=\int \bar{b}(x,u)\varphi_{\mathbf{k}}(u)\,w(u)\,du,~
\beta_{\mathbf{k}}(x)   :=\int \bar{\sigma}(x,u)\varphi_{\mathbf{k}}(u)\,w(u)\,du.
\end{align}
With this notation, \eqref{eq:mlmceulerlineartargetinteractprojnew} becomes
 \begin{eqnarray}\label{eq:defxtminteract_project_K}
d\tilde{X}^{K}_t &=&\sum_{|\mathbf{k}|\leq K}\alpha
_{\mathbf{k}}(\tilde{X}^{K}_t)\gamma^{\tilde{X}^K}_{\mathbf{k}}(t)\, dt + \sum_{|\mathbf{k}|\leq K}\beta
_{\mathbf{k}}(\tilde{X}^{K}_t)\gamma^{\tilde{X}^K}_{\mathbf{k}}(t)\, dW_t\nonumber
\\
&=& A_K(t, \tilde{X}^{K}_t)\, dt + B_K(t, \tilde{X}^{K}_t)\, dW_t,\quad \tilde{X}^K_0\sim \mu_0,
\end{eqnarray}
with \(|\mathbf{k}|=\max\{k_1,\ldots,k_d\}\) where  \(\mathbf{k}=(k_1,\ldots,k_d)\).  The equation \eqref{eq:defxtminteract_project_K} is a special case of \eqref{eq:mlmceulerlineartargetinteract}  referred to as McKean-Vlasov SDEs with non-interacting kernels;  the measure dependence is only through $K$ expected values and the corresponding algorithm requires only $O(KN)$ numerical operations plus integrations at each time step. If $K$ is much smaller than $N$ and integration cost is negligible, we obtain a substantial cost reduction. Also, this form allows for iterative MLMC methodology to be employed.
We show that  iterative MLMC via projections would reduce an extra one order of magnitude compared to the results in \cite{belomestny2018projected}. On the one hand, this reduction of an extra one order is due to the iterative MLMC across level of time-discretisation. On the other hand, the antithetic MLMC in \cite{szpruch2019antithetic} renders a massive improvement from balancing cost between the number of ensembles and number of particles in each ensemble since cost of $\eqref{eq:mlmceulerlineartargetinteract}$ is linear in the number of ensembles and quadratic in the number of particles. It would not help for the projected process \eqref{eq:defxtminteract_project_K} because the number of ensembles and the number of particles both are linear in cost, and there is no MLMC in time discretisation.  Finally, we present fully implementable algorithm that can exploit smoothness of the coefficients in \eqref{eq:mlmceulerlineartargetinteract} and in the best case scenario, that is, if all the coefficient functions are infinitely smooth, the complexities reduce to one of order  $O(\varepsilon^{-2} |\log(\epsilon)|^{4})$ for approximation of expectations (see other MLMC approaches in \cite{2016arXiv161009934H} and \cite{2015arXiv150802299R}) and $O(\epsilon^{-2}|\log\epsilon|^5)$ for density estimation. 
\section{Iterative MLMC algorithm with projected coefficients}\label{sec:desmlmc}

\textbf{Projections.} Recall that $(\varphi_{\mathbf{k}},$ $\mathbf{k}\in \mathbb{N}_0^d)$ is a total orthonormal
system in \thinspace$L_{2}(\mathbb{R}^{d},w)$ for some  weight
function \(w(x)>0\), \(x\in \mathbb{R}^d\) and $\alpha_{\mathbf{k}}(x)$ and $\beta_{\mathbf{k}}(x) $ defined in \eqref{abd}, so that
\begin{align*}
\int_{\mathbb{R}^d}\bar{b}(x,y)\mu_{t}(y)dy   =\sum_{\mathbf{k}\in \mathbb{N}_0^d}\alpha
_{\mathbf{k}}(x)\gamma_{\mathbf{k}}(t),~\int_{\mathbb{R}^d}\bar{\sigma}(x,y)\mu_{t}(y)dy  =\sum_{\mathbf{k}\in \mathbb{N}_0^d}\beta
_{\mathbf{k}}(x)\gamma_{\mathbf{k}}(t),
\end{align*}
and \eqref{eq:mlmceulerlineartargetinteract} transforms to
\begin{equation}\label{eq:defxtminteract_project_inf}
dX_t =\sum_{\mathbf{k}\in \mathbb{N}_0^d}\alpha
_{\mathbf{k}}(x)\gamma_{\mathbf{k}}(t)\, dt + \sum_{\mathbf{k}\in \mathbb{N}_0^d}\beta
_{\mathbf{k}}(x)\gamma_{\mathbf{k}}(t)\, dW_t.
\end{equation}
 The idea of the projection estimation consists in replacing the infinite sums in \eqref{eq:defxtminteract_project_inf} by finite sums of  the first at most $d\cdot(K+1)$ terms, i.e. \eqref{eq:defxtminteract_project_K}. If the coefficients \(\gamma_\mathbf{k}\) decay fast as \(|\mathbf{k}|\to \infty,\) one expects that the corresponding approximations \(A_\mathbf{k}\) and \(B_\mathbf{k}\) in \eqref{eq:defxtminteract_project_K} are close to the original coefficients. The  associated particle system for \eqref{eq:defxtminteract_project_K} is given by
 \begin{align}\label{eq:defxtminteract_project_Keuler}
 \bar{X}^{i,N,K}_{t_{s+1}} &=\bar{X}^{i,N,K}_{t_s}+\sum_{|\mathbf{k}|\leq K}\alpha
_{\mathbf{k}}(\bar{X}^{i,N,K}_{t_s})\hat{\gamma}_{\mathbf{k}}(t_s)\, (t_{s+1}-t_s) + \sum_{|\mathbf{k}|\leq K}\beta
_{\mathbf{k}}(\bar{X}^{i,N,K}_{t_s})\hat{\gamma}_{\mathbf{k}}(t_s)\,(W_{t_{s+1}}^{i}-W_{t_s}^{i}),
 \end{align}
 where $\hat{\gamma}_{\mathbf{k}}(t) = \frac{1}{N}\sum_{j=1}^N\varphi_{\mathbf{k}}(\bar{X}^{j,N,K}_t)$.
We work under the following assumptions:
\begin{description}
\item[\textbf{(Ker-Reg)}{ }]  $\bar{b} \in C^{2,s}_{b,b}(\bbR^d\times\bbR^{d}, \bbR^d)$ and $\bar{\sigma}\in C^{2,s}_{b,b}(\bbR^d\times\bbR^{d}, \bbR^{d\otimes r})$, where $s>2$. 
\item[\textbf{($\mu_0$-$L_p$)}{ }]  There exists constant $\rho_{0}>0,$ $\rho_{1}>0$ such that
\[
\mu_{0}(x)\lesssim\exp\left(  -\rho_{0}|x|^{\rho_{1}}\right)  ,\quad
|x|\rightarrow\infty
\]

\item[(AF)] The basis functions $(\varphi_{\mathbf{k}}) \in C^{2}_{b}( \mathbb{R}^d, \mathbb{R}^d)$ fulfil there exists constants $L_{\mathbf{k},\varphi
},D_{\mathbf{k},\varphi}$
 such that
\[
\left\vert \varphi_{\mathbf{k}}(z)-\varphi_{\mathbf{k}}(z^{\prime})\right\vert \leq
L_{\mathbf{k},\varphi}\left\vert z-z^{\prime}\right\vert ,\quad\left\vert \varphi
_{\mathbf{k}}(z)\right\vert \leq D_{\mathbf{k},\varphi},\quad \mathbf{k}\in \mathbb{N}_0^d 
\]
for all $z,z^{\prime}\in\mathbb{R}^{d}$.

\item[(AC)]  There are positive constants \(L_\varphi, D_\varphi, A_{\alpha},A_{\beta},\) \(B_{\alpha},B_{\beta}\) such that 
\begin{align*}
\sup_{x\in\mathbb{R}^{d}}\left\vert \alpha_{\mathbf{k}}(x)\right\vert  &  \leq
A_{\mathbf{k},\alpha}(1+\left\vert x\right\vert ) \text{ \ \ with }\
\text{ \ }\sum_{\mathbf{k}\in \mathbb{N}_0^d}A_{\mathbf{k},\alpha} \leq A_{\alpha}~\text{and}~\sum_{\mathbf{k}\in \mathbb{N}_0^d} L_{\mathbf{k},\varphi}A_{\mathbf{k},\alpha}\leq L_{\varphi}A_{\alpha}, \\
\sup_{x\in\mathbb{R}^{d}}\left\vert \beta_{\mathbf{k}}(x)\right\vert  &  \leq
A_{\mathbf{k},\beta}(1+\left\vert x\right\vert )\text{ \ \ with}\ 
\text{ \ }\sum_{\mathbf{k}\in \mathbb{N}_0^d}A_{\mathbf{k},\beta}   \leq A_{\beta}~\text{and}~\sum_{\mathbf{k}\in \mathbb{N}_0^d}L_{\mathbf{k},\varphi}A_{\mathbf{k},\beta}\leq L_{\varphi}A_{\beta},
\end{align*}
and 
\begin{align*}
\sup_{x,x^{\prime}\in\mathbb{R}^{d}}\frac{|\alpha_{\mathbf{k}}(x)-\alpha_{\mathbf{k}}(x^{\prime
})|}{|x-x^{\prime}|} &  \leq B_{\mathbf{k},\alpha}\text{ \ \ with}\  
\text{ \ }\sum_{\mathbf{k}\in \mathbb{N}_0^d}B_{\mathbf{k},\alpha}  \leq B_{\alpha}~\text{and}~\sum_{\mathbf{k}\in \mathbb{N}_0^d}D_{\mathbf{k},\varphi}B_{\mathbf{k},\alpha}\leq D_{\varphi}B_{\alpha} 
\\
\sup_{x,x^{\prime}\in\mathbb{R}^{d}}\sum_{\mathbf{k}\in \mathbb{N}_0^d}^{\infty}\frac{|\beta
_{\mathbf{k}}(x)-\beta_{\mathbf{k}}(x^{\prime})|}{|x-x^{\prime}|} &  \leq B_{\mathbf{k},\beta}\text{ \ \ with}\ 
\text{ \ } \sum_{\mathbf{k}\in \mathbb{N}_0^d} B_{\mathbf{k},\beta}   \leq B_{\beta}~\text{and}~\sum_{\mathbf{k}\in \mathbb{N}_0^d}D_{\mathbf{k},\varphi}B_{\mathbf{k},\beta}\leq D_{\varphi}B_{\beta}.
\end{align*}
 \item[(AD)] There is a constant $\gamma^\circ$ satisfying 
 \begin{eqnarray*}\label{eq:gammadecayAD}
|\gamma_{\mathbf{k}}(t)|\lesssim \exp(-\gamma^\circ|\mathbf{k}|), \quad |\mathbf{k}|\to \infty,
\end{eqnarray*}

\end{description}
Note that if \hkreg \,  holds, then 
\begin{itemize}
\item[\hLip]\hspace{0.4cm}the kernels $\bar{b}$ and $\bar{\sigma}$ are globally Lipschitz,  i.e. 
for all $x_1,x_2, y_1, y_2\in\mathbb{R}^d$, there exists a constant $L$ such that
$$ | \bar{b}(x_1,y_1) -\bar{b}(x_2,y_2) |  +  \| \bar{\sigma}(x_1,y_1) - \bar{\sigma}(x_2,y_2) \|  \leq L ( |x_1-x_2| + |y_1-y_2|). $$
\end{itemize}
\begin{remark}
All involved Lipschitz constants decay when $|\mathbf{k}|\rightarrow \infty$  (this is crucial for rigorous asymptotic analysis). Note that (AF) and \textbf{(Ker-Reg)}{ }  do not imply (AC). (AD) is satisfied if for example the coefficients $\bar{b}$ and $\bar{\sigma}$ are infinitely smooth with respect to the system $(\varphi_{\mathbf{k}})_{\mathbf{k}\in \mathbb{N}^d_0 }$. Also,  $\gamma^{\circ}$ may grow at most polynomially in $d$ and this would lead to polynomial dependence of complexity on $d$. We implicitly assume $\alpha
_{\mathbf{k}}$ and $\beta_{\mathbf{k}}$ are known in closed form or can be computed cheaply.  We denote by $C^{p}_{b}(\mathbb{R}^m, \mathbb{R}^d)$ the set of functions from $\mathbb{R}^m$ to $\mathbb{R}^d$ that are  continuously $p$ times differentiable such that the partial derivatives are bounded, and  $C^{p,q}_{b,b}(\mathbb{R}^m \times \mathbb{R}^n, \mathbb{R}^d)$ the set of functions from $\mathbb{R}^m \times \mathbb{R}^n $ to $\mathbb{R}^d$ that are continuously $p$ times differentiable in the first argument and continuously $q$ times differentiable in the second argument such that the corresponding partial derivatives are bounded. 
\end{remark}

From \cite{antonelli2002rate} and \cite{sznitman1991topics}, the standard assumptions \textbf{(Ker-Reg)}{ } and \textbf{($\mu_0$-$L_p$)}{ } ensure that there exists a unique strong solution of \eqref{eq:mlmceulerlineartargetinteract} and the integrability of $X_t$ holds, i.e. for any $p\geq 2$, $\sup_{0\leq t\leq T}\bbE|X_t|^p <\infty.$

\par
\textbf{Analysis of projections}\label{sec:Projections}
First, we have fairly general conditions (AF) and (AC) that guarantee that the coefficients of projected MVSDEs \eqref{eq:defxtminteract_project_inf} are globally Lipschitz and of linear growth. Next we check conditions (AF) and (AC) for the case of  normalised Hermite polynomials.

\begin{example}
\label{exm:hermite}
The (normalised) Hermite polynomial of order $j$ is given, for $j\geq0$, by
\[
\overline{H}_{j}(x)=c_j (-1)^{j}e^{x^{2}}\frac{d^{j}}{dx^{j}}(e^{-x^{2}}),\quad c_{j}=\left(  2^{j}j!\sqrt{\pi
}\right) ^{-1/2}.
\]
These polynomials are orthonormal with respect to the weight function
$w(x)=e^{-x^{2}}$ and satisfy: $\int_{{\mathbb{R}}}\overline{H}_{j}(x)\overline{H}_{\ell}(x)e^{-x^{2}%
}dx=\delta_{j,\ell}$.  Set 
\begin{eqnarray}
\label{eq:herm_func}
\varphi_{\mathbf{k}}(u)= \overline{H}_{k_1}(u)e^{-u^{2}/2}\cdot\ldots\cdot \overline{H}_{k_d}(u)e^{-u^{2}/2}
\end{eqnarray}
where \(\mathbf{k}=(k_1,\ldots,k_d)\in \mathbb{N}_0^d.\)
Then the system \((\varphi_{\mathbf{k}})_{\mathbf{k}\in \mathbb{N}^d_0 }\) is a complete orthonormal system in \(L_2(\mathbb{R}^d)\) fulfilling  the assumption (AF) with \(D_{\mathbf{k},\varphi}\) and \(L_{\mathbf{k},\varphi}\) uniformly bounded in \(\mathbf{k}.\) Suppose that \(\mu_t\in L_2(\mathbb{R}^d)\) and \(b(x,\cdot),\sigma(x,\cdot)\in L_2(\mathbb{R}^d).\)
Then the density $\mu_{t}$ can be developed in the Hermite
basis $\mu_{t}(u)=\sum_{\mathbf{j}\in\mathbb{N}^d_0}\gamma_{\mathbf{j}}(t)\varphi_{\mathbf{j}}(u)$ where $\gamma
_{\mathbf{j}}(t)=\int_{{\mathbb{R}^d}}\mu_{t}(x)\varphi_{\mathbf{j}}(x)dx$.
From \cite{belomestny2018projected}, one can show (AC) is satisfied.
\end{example}

\subsection{Weak error analysis} \label{sec:weamain}
We then reveal the order of weak convergence of \eqref{eq:defxtminteract_project_K} to \eqref{eq:mlmceulerlineartargetinteract} in the lemma below. We begin this analysis by defining  $\mathcal{X}^{s,x}$ via
\begin{equation}\label{eq:weprocess}
\mathcal{X}^{s,x}_t = x + \int_s^t \sum_{\mathbf{k}\in \mathbb{N}_0^d}\alpha
_{\mathbf{k}}(\mathcal{X}^{s,x}_u)\gamma_{\mathbf{k}}(u)\,du +\int_s^t \sum_{\mathbf{k}\in \mathbb{N}_0^d}\beta_{\mathbf{k}}(\mathcal{X}^{s,x}_u)\gamma_{\mathbf{k}}(u) \,dW_u,
\end{equation}
where $\alpha_{\mathbf{k}}(\cdot)$ and $\beta_{\mathbf{k}}(\cdot)$ are defined in \eqref{abd}. Notice that \eqref{eq:weprocess}, unlike \eqref{eq:mlmceulerlineartargetinteract}, is a Markov process. For $P \in C^{2}_{b}( \mathbb{R}^d, \mathbb{R}^d)$ and  $t \in [0,T]$, with assumptions listed in Section \ref{sec:desmlmc}, we know from \cite{szpruch2019},
\[
\int_{\mathbb{R}^d}\bbE \big[ P( \mathcal{X}^{0,x}_t) \big] \, \mu_0(dx)  =  \bbE \big[ \bbE[P(X_t) | X_0  ] \big],
\]
which we will use in Theorem \ref{lm:wkmain}. We consider the function

\begin{align}\label{eq:defvinteract}
v(s,x):=\mathbb{E}[ P(\mathcal{X}_t^{s,x})], \quad (s,x)\in[0,t]\times \mathbb{R}^d.
\end{align}
We aim to show that $v(s,x)\in C^{1,2}$. The lemma \ref{lm:regularityBinteract} gives the first step. 
\begin{lemma}\label{lm:regularityBinteract}
Assume  \textbf{($\mu_0$-$L_p$)}{ } \, and \textbf{(Ker-Reg)}{ }.   Then 
\[
	\sum_{\mathbf{k}\in \mathbb{N}_0^d}\alpha
_{\mathbf{k}}(\cdot)\gamma_{\mathbf{k}}(\cdot)\in C^{2,1}_{b,b}(\mathbb{R}^d \times [0,T],\mathbb{R}^d)~\text{and}~\sum_{\mathbf{k}\in \mathbb{N}_0^d}\beta_{\mathbf{k}}(\cdot)\gamma_{\mathbf{k}}(\cdot)\in C^{2,1}_{b,b}(\mathbb{R}^d \times [0,T],\mathbb{R}^{d\otimes r}).
\]
\end{lemma}
\begin{proof}
The proof can be carried out along the same lines as in \cite{szpruch2019}.
\end{proof}
Next, the following lemma verifies the property of $v(s,x)$ that is needed:
\begin{lemma} \label{eq:regularityofderivativeofvy}${}$ Assume (AC), \textbf{(Ker-Reg)}{ }, and \textbf{($\mu_0$-$L_p$)}{ }. Then for any $(s,x) \in [0,t] \times \mathbb{R}^d$,  $(i,j)\in\{1,\ldots,d\}^2$, $\partial_{s} v (s,x)$ is continuous and 
\[ \tag{\hvdiffs}
       \|\partial_{x_i} v (s,x) \|_{\infty}   + \|\partial^2_{x_i, x_j} v  (s,x) \|_\infty \leq L. 
\]  
\end{lemma}
\begin{proof}
We follow the same lines as in \cite{szpruch2019} to show this.
\end{proof}
Appying the Feynman-Kac theorem (\cite{MR601776}), one can find that  $v(\cdot,\cdot)$ satisfies the following Cauchy problem component-wise, 
\begin{equation}\label{eq:proofalphafinalnonintinteractnon}
 \left\{
\begin{array}{rl}
        \partial_s v (s,x)  + &\displaystyle\frac{1}{2}\smsec  \Big( \sum_{\mathbf{k}\in \mathbb{N}_0^d}\beta_{\mathbf{k}}(x)\gamma_{\mathbf{k}}(s) ( \sum_{\mathbf{k}\in \mathbb{N}_0^d}\beta_{\mathbf{k}}(x)\gamma_{\mathbf{k}}(s) )^T\Big)_{ij}  \partial^2_{x_i,x_j} v (s,x) \\
       +\displaystyle\sum_{j=1}^d \Big( &\sum_{\mathbf{k}\in \mathbb{N}_0^d}\alpha
_{\mathbf{k}}(x)\gamma_{\mathbf{k}}(s) \Big)_j  \partial_{x_j} v  (s,x) =0, \quad (s,x)\in[0,t]\times \mathbb{R}^d,\\
        v(t,x) &= P(x). \\
\end{array} 
\right.
\end{equation}

\begin{theorem}\label{lm:wkmain}
Let $P \in C^{2}_{b}( \mathbb{R}^d, \mathbb{R}^d)$, and processes $(X_t)$ and $({\tilde{X}_t^K})$ be defined in \eqref{eq:mlmceulerlineartargetinteract} and \eqref{eq:defxtminteract_project_K}, respectively. Assume \textbf{(Ker-Reg)}, \textbf{($\mu_0$-$L_p$)}, (AC) and (AF) hold. Then for any $t\in(0,T]$,
\[
	|\bbE[P(\tilde{X}_{t}^K)]-\bbE[P({X}_{t})]|\lesssim\sum_{|\mathbf{k}|>K}\sup_{0< t\leq T}|\gamma_{\mathbf{k}}(t)|.
\]
\end{theorem}
\begin{proof}
From definition of $v(\cdot,\cdot)$ in \eqref{eq:defvinteract}, we compute that
\begin{eqnarray}
\bbE[v(0,X_0)]=\int_{\mathbb{R}^d} v(0,x) \, \mu_0(dx) &  = & \int_{\mathbb{R}^d}\bbE \big[ P( \mathcal{X}^{0,x}_t) \big] \, \mu_0(dx) \nonumber \\
& = & \bbE \big[ \bbE[P( X_t) | X_0  ] \big].  \nonumber 
\end{eqnarray}
The Feynman-Kac theorem, and the fact that $\mu_0 = \mu^{\tilde{X}^K}_0$ give 
\begin{align*}
	\bbE[P(\tilde{X}_{t}^K)] - \bbE[P(X_t)] 
	&= \bbE[v(t,\tilde{X}_{t}^K)] - \bbE[v(0, {X}_{0})] \\
	&= \bbE[v(t,\tilde{X}_{t}^K)] - \bbE[v(0, \tilde{X}_{0}^K)].
\end{align*}  By It\^o's formula, together with definitions of $\alpha
_{\mathbf{k}}(\cdot)$ and $\beta_{\mathbf{k}}(\cdot)$,
\begin{eqnarray}
	&&\bbE[v(t,\tilde{X}_{t}^K)] - \bbE[v(0, \tilde{X}_{0}^K)] \nonumber \\ 
	& = & \bbE\bigg[\int_{0}^{t}\bigg(\partial_t v(s, \tilde{X}_{s}^K) +\sum_{j=1}^d \partial_{x_j} v (s,\tilde{X}_{s}^K)\sum_{|\mathbf{k}|\leq K}\alpha
_{\mathbf{k}}(\tilde{X}_{s}^K)\gamma_{\mathbf{k}}^{\tilde{X}^K}(s)+\dfrac{1}{2}\smsec\partial^2_{x_i,x_j} v  (s, \tilde{X}_{s}^K)\nonumber\\
&&\cdot \overline{a}_{ij}(\tilde{X}_{s}^K, \gamma_{\cdot}^{\tilde{X}^K}(s) )\bigg)ds+ \int_{0}^{t} \smfir\smfi
	\partial_{x_j} v   (s,\tilde{X}_{s}^K) \sum_{|\mathbf{k}|\leq K}\beta
_{\mathbf{k},{ji}}(\tilde{X}_{s}^K)\gamma^{\tilde{X}^K}_{\mathbf{k}}(s)dW_s^{(i)} \, \bigg], \nonumber 
\end{eqnarray}
where $\overline{a}(x,y)=(\sum_{|\mathbf{k}|\leq K}\beta_{\mathbf{k}}(x)y_{\mathbf{k}})\cdot(\sum_{|\mathbf{k}|\leq K}\beta_{\mathbf{k}}(x)y_{\mathbf{k}})^T$. (AC) and the condition \hvdiffs~implies
\[
	\bbE\left[\int_{0}^{t} \smfir\smfi	\partial_{x_j} v   (s,\tilde{X}_{s}^K)\sum_{|\mathbf{k}|\leq K}\beta
_{\mathbf{k},{ji}}(\tilde{X}_{s}^K)\gamma^{\tilde{X}^K}_{\mathbf{k}}(s)dW_s^{(i)} \right]=0.
\]
Subsequently, using the fact that $v(\cdot,\cdot)$ satisfies PDE \eqref{eq:proofalphafinalnonintinteractnon} at $(s,\tilde{X}_{s}^K)$, we have
\begin{align*}
 \partial_s v (s,\tilde{X}_{s}^K)  + &\displaystyle\frac{1}{2}\smsec  \Big( \sum_{\mathbf{k}\in \mathbb{N}_0^d}^{\infty}\beta_{\mathbf{k}}(\tilde{X}_{s}^K)\gamma_{\mathbf{k}}(s) \Big( \sum_{\mathbf{k}\in \mathbb{N}_0^d}\beta_{\mathbf{k}}(\tilde{X}_{s}^K)\gamma_{\mathbf{k}}(s) \Big)^T\Big)_{ij}  \partial^2_{x_i,x_j} v (s,\tilde{X}_{s}^K) \\
       +\displaystyle\sum_{j=1}^d \Big( &\sum_{\mathbf{k}\in \mathbb{N}_0^d}\alpha
_{\mathbf{k}}(\tilde{X}_{s}^K)\gamma_{\mathbf{k}}(s) \Big)_j  \partial_{x_j} v  (s,\tilde{X}_{s}^K) =0.
\end{align*}
Then we rewrite the coefficient functions by splitting each infinite sum into $\sum_{|\mathbf{k}|\leq K}$ and $\sum_{|\mathbf{k}|> K}$ ( $\mathbf{k}_1$ or $\mathbf{k}_2$ is dummy variable for $\mathbf{k}$)  and we obtain
\begin{align*}
 \partial_s v (s,\tilde{X}_{s}^K)  + &\displaystyle\frac{1}{2}\smsec  \Big( (\sum_{|\mathbf{k_1}|\leq K}\beta_{\mathbf{k_1}}(\tilde{X}_{s}^K)\gamma_{\mathbf{k_1}}(s) +\sum_{|\mathbf{k_1}|>K}\beta_{\mathbf{k_1}}(\tilde{X}_{s}^K)\gamma_{\mathbf{k_1}}(s))\cdot\\
&( \sum_{|\mathbf{k_2}|\leq K}\beta_{\mathbf{k_2}}(\tilde{X}_{s}^K)\gamma_{\mathbf{k_2}}(s) +\sum_{|\mathbf{k_2}|>K}\beta_{\mathbf{k_2}}(\tilde{X}_{s}^K)\gamma_{\mathbf{k_2}}(s))^T\Big)_{ij}  \partial^2_{x_i,x_j} v (s,\tilde{X}_{s}^K) \\
       +\displaystyle\sum_{j=1}^d \Big( &\sum_{|\mathbf{k}|\leq K}\alpha
_{\mathbf{k}}(\tilde{X}_{s}^K)\gamma_{\mathbf{k}}(s) +\sum_{|\mathbf{k}|> K}\alpha
_{\mathbf{k}}(\tilde{X}_{s}^K)\gamma_{\mathbf{k}}(s)\Big)_j  \partial_{x_j} v  (s,\tilde{X}_{s}^K) =0.
\end{align*}
Further expanding the coefficient for $\partial^2_{x_i,x_j} v (s,\tilde{X}_{s}^K)$ in PDE, we have (four terms)
\begin{align*}
 &\Big( (\sum_{|\mathbf{k_1}|\leq K}\beta_{\mathbf{k_1}}(\tilde{X}_{s}^K)\gamma_{\mathbf{k_1}}(s) +\sum_{|\mathbf{k_1}|>K}\beta_{\mathbf{k_1}}(\tilde{X}_{s}^K)\gamma_{\mathbf{k_1}}(s))( \sum_{|\mathbf{k_2}|\leq K}\beta_{\mathbf{k_2}}(\tilde{X}_{s}^K)\gamma_{\mathbf{k_2}}(s) +\sum_{|\mathbf{k_2}|>K}\beta_{\mathbf{k_2}}(\tilde{X}_{s}^K)\gamma_{\mathbf{k_2}}(s))^T\Big)_{ij}\\
 &=\Big(\sum_{|\mathbf{k_1}|\leq K}\beta_{\mathbf{k_1}}(\tilde{X}_{s}^K)\gamma_{\mathbf{k_1}}(s)\cdot (\sum_{|\mathbf{k_2}|\leq K}\beta_{\mathbf{k_2}}(\tilde{X}_{s}^K)\gamma_{\mathbf{k_2}}(s))^T+\sum_{|\mathbf{k_1}|\leq K}\beta_{\mathbf{k_1}}(\tilde{X}_{s}^K)\gamma_{\mathbf{k_1}}(s)\cdot(\sum_{|\mathbf{k_2}|>K}\beta_{\mathbf{k_2}}(\tilde{X}_{s}^K)\gamma_{\mathbf{k_2}}(s))^T\\
 &+\sum_{|\mathbf{k_1}|>K}\beta_{\mathbf{k_1}}(\tilde{X}_{s}^K)\gamma_{\mathbf{k_1}}(s)\cdot(\sum_{|\mathbf{k_2}|\leq K}\beta_{\mathbf{k_2}}(\tilde{X}_{s}^K)\gamma_{\mathbf{k_2}}(s))^T+\sum_{|\mathbf{k_1}|>K}\beta_{\mathbf{k_1}}(\tilde{X}_{s}^K)\gamma_{\mathbf{k_1}}(s)\cdot(\sum_{|\mathbf{k_2}|>K}\beta_{\mathbf{k_2}}(\tilde{X}_{s}^K)\gamma_{\mathbf{k_2}}(s))^T\Big)_{ij}
\end{align*}
Taking this expansion, we compute
\allowdisplaybreaks
\begin{align*}
	\bbE[&v(t,\tilde{X}_{t}^K)] - \bbE[v(0, \tilde{X}_{0}^K)]  \\ 
	& = \bbE\bigg[\int_{0}^{t}\bigg[\sum_{j=1}^d \partial_{x_j} v (s,\tilde{X}_{s}^K)\bigg(\sum_{|\mathbf{k}|\leq K}\alpha
_{\mathbf{k}}(\tilde{X}_{s}^K)(\gamma_{\mathbf{k}}^{\tilde{X}^K}(s)-\gamma_{\mathbf{k}}(s))-\sum_{|\mathbf{k}|> K}\alpha
_{\mathbf{k}}(\tilde{X}_{s}^K)\gamma_{\mathbf{k}}(s)\bigg)\\
&+\dfrac{1}{2}\smsec\partial^2_{x_i,x_j} v  (s, \tilde{X}_{s}^K)\bigg((\overline{a}_{ij}(\tilde{X}_{s}^K, \gamma_{\cdot}^{\tilde{X}^K}(s)) - \overline{a}_{ij}(\tilde{X}_{s}^K,\gamma_{\cdot}(s))) -\sum_{|\mathbf{k}_1|> K}\sum_{|\mathbf{k}_2|\leq K}(\beta_{\mathbf{k}_1,i}(\tilde{X}_{s}^K)\gamma_{\mathbf{k}_1}(s))\\
&\cdot(\beta_{\mathbf{k}_2,j}(\tilde{X}_{s}^K)
\gamma_{\mathbf{k}_2}(s))^T-\sum_{|\mathbf{k}_1|\leq K}\sum_{|\mathbf{k}_2|> K}(\beta_{\mathbf{k}_1,i}(\tilde{X}_{s}^K)\gamma_{\mathbf{k}_1}(s))\cdot(\beta_{\mathbf{k}_2,j}(\tilde{X}_{s}^K)\gamma_{\mathbf{k}_2}(s))^T\\
&-\sum_{|\mathbf{k}_1|> K}\sum_{|\mathbf{k}_2|> K}(\beta_{\mathbf{k}_1,i}(\tilde{X}_{s}^K)\gamma_{\mathbf{k}_1}(s))\cdot(\beta_{\mathbf{k}_2,j}(\tilde{X}_{s}^K)\gamma_{\mathbf{k}_2}(s))^T\bigg)\bigg]ds\bigg]\\
&=\int_0^t\bbE[\sum_{i=1}^5R_i(s)]ds,
\end{align*} 
where
\begin{align*}
&R_1(s) = \sum_{j=1}^d \partial_{x_j} v (s,\tilde{X}_{s}^K)\bigg(\sum_{|\mathbf{k}|\leq K}\alpha
_{\mathbf{k}}(\tilde{X}_{s}^K)(\gamma_{\mathbf{k}}^{\tilde{X}^K}(s)-\gamma_{\mathbf{k}}(s))-\sum_{|\mathbf{k}|> K}\alpha
_{\mathbf{k}}(\tilde{X}_{s}^K)\gamma_{\mathbf{k}}(s)\bigg)\nonumber\\
&R_2(s) = \dfrac{1}{2}\smsec\partial^2_{x_i,x_j} v  (s, \tilde{X}_{s}^K)\bigg(\overline{a}_{ij}(\tilde{X}_{s}^K, \gamma_{\cdot}^{\tilde{X}^K}(s)) - \overline{a}_{ij}(\tilde{X}_{s}^K,\gamma_{\cdot}(s))\bigg)\nonumber\\
&R_3(s) =- \dfrac{1}{2}\smsec\partial^2_{x_i,x_j} v  (s, \tilde{X}_{s}^K)\bigg(\sum_{|\mathbf{k}_1|> K}\sum_{|\mathbf{k}_2|\leq K}(\beta_{\mathbf{k}_1,i}(\tilde{X}_{s}^K)\gamma_{\mathbf{k}_1}(s))\cdot(\beta_{\mathbf{k}_2,j}(\tilde{X}_{s}^K)
\gamma_{\mathbf{k}_2}(s))^T\bigg)\nonumber\\
&R_4(s) = -\dfrac{1}{2}\smsec\partial^2_{x_i,x_j} v  (s, \tilde{X}_{s}^K)\bigg(\sum_{|\mathbf{k}_1|\leq K}\sum_{|\mathbf{k}_2|> K}(\beta_{\mathbf{k}_1,i}(\tilde{X}_{s}^K)\gamma_{\mathbf{k}_1}(s))\cdot(\beta_{\mathbf{k}_2,j}(\tilde{X}_{s}^K)\gamma_{\mathbf{k}_2}(s))^T\bigg)\nonumber\\
&R_5(s) =  -\dfrac{1}{2}\smsec\partial^2_{x_i,x_j} v  (s, \tilde{X}_{s}^K)\bigg(\sum_{|\mathbf{k}_1|> K}\sum_{|\mathbf{k}_2|> K}(\beta_{\mathbf{k}_1,i}(\tilde{X}_{s}^K)\gamma_{\mathbf{k}_1}(s))\cdot(\beta_{\mathbf{k}_2,j}(\tilde{X}_{s}^K)\gamma_{\mathbf{k}_2}(s))^T\bigg)
\end{align*}
\paragraph*{Error $R_1$:}
\hvdiffs~(bounded derivatives for $v$)  gives
\begin{align*}
|\bbE[R_1(s)]| &\lesssim \bbE\bigg[\sum_{j=1}^d|\partial_{x_j} v (s,\tilde{X}_{s}^K)|\cdot|\sum_{|\mathbf{k}|\leq K}\alpha
_{\mathbf{k}}(\tilde{X}_{s}^K)(\gamma_{\mathbf{k}}^{\tilde{X}^K}(s)-\gamma_{\mathbf{k}}(s))-\sum_{|\mathbf{k}|> K}\alpha
_{\mathbf{k}}(\tilde{X}_{s}^K)\gamma_{\mathbf{k}}(s)|\bigg]\\
&\lesssim \bbE|\sum_{|\mathbf{k}|\leq K}\alpha
_{\mathbf{k}}(\tilde{X}_{s}^K)(\gamma_{\mathbf{k}}^{\tilde{X}^K}(s)-\gamma_{\mathbf{k}}(s))-\sum_{|\mathbf{k}|> K}\alpha
_{\mathbf{k}}(\tilde{X}_{s}^K)\gamma_{\mathbf{k}}(s)|\\
&\lesssim \bbE|\sum_{|\mathbf{k}|\leq K}\alpha
_{\mathbf{k}}(\tilde{X}_{s}^K)(\gamma_{\mathbf{k}}^{\tilde{X}^K}(s)-\gamma_{\mathbf{k}}(s))|+\bbE|\sum_{|\mathbf{k}|> K}\alpha
_{\mathbf{k}}(\tilde{X}_{s}^K)\gamma_{\mathbf{k}}(s)|\\
&(\text{non-randomness property of $\gamma_{\mathbf{k}}^{\tilde{X}^K}(s)$ and $\gamma_{\mathbf{k}}(s)$})\\
&\lesssim \sup_{\mathbf{k}}|\gamma_{\mathbf{k}}^{\tilde{X}^K}(s)-\gamma_{\mathbf{k}}(s)|\cdot\sum_{|\mathbf{k}|\leq K}\bbE|\alpha_{\mathbf{k}}(\tilde{X}_{s}^K)|+\sup_{\mathbf{k}}\bbE|\alpha
_{\mathbf{k}}(\tilde{X}_{s}^K)|\cdot\sum_{|\mathbf{k}|> K}|\gamma_{\mathbf{k}}(s)|.
\end{align*}
By (AC) and the fact $\tilde{X}_{s}^K$ is integrable, it implies $\sum_{|\mathbf{k}|\leq K}\bbE|\alpha_{\mathbf{k}}(\tilde{X}_{s}^K)|\lesssim\sum_{|\mathbf{k}|\in N^d_0}A_{\mathbf{k},\varphi}\leq A_{\alpha}$ and $\sup_{\mathbf{k}}\bbE|\alpha_{\mathbf{k}}(\tilde{X}_{s}^K)|\lesssim A_{\alpha}$ . Then we show
\begin{align}\label{eq:wkm1}
|\bbE[R_1(s)]|\lesssim \sup_{\mathbf{k}}|\gamma_{\mathbf{k}}^{\tilde{X}^K}(s)-\gamma_{\mathbf{k}}(s)|+\sum_{|\mathbf{k}|> K}|\gamma_{\mathbf{k}}(s)|.
\end{align}

\paragraph*{Error $R_2$:}
Again, \hvdiffs~(bounded second derivatives for $v$)  gives
\begin{align*}
|\bbE[R_2(s)]| &\lesssim \bbE\bigg[\dfrac{1}{2}\smsec|\partial^2_{x_i,x_j} v  (s, \tilde{X}_{s}^K)|\cdot|\overline{a}_{ij}(\tilde{X}_{s}^K, \gamma_{\cdot}^{\tilde{X}^K}(s)) - \overline{a}_{ij}(\tilde{X}_{s}^K,\gamma_{\cdot}(s))|\bigg]\\
&\lesssim \bbE|\overline{a}_{ij}(\tilde{X}_{s}^K, \gamma_{\cdot}^{\tilde{X}^K}(s)) - \overline{a}_{ij}(\tilde{X}_{s}^K,\gamma_{\cdot}(s))|\\
&(\text{non-randomness property of $\gamma_{\mathbf{k}}^{\tilde{X}^K}(s)$ and $\gamma_{\mathbf{k}}(s)$} )\\
&\lesssim \sum_{|\mathbf{k}_1|\leq K}\sum_{|\mathbf{k}_2|\leq K}|\gamma_{\mathbf{k}_1}^{\tilde{X}^K}(s)M_{\beta_{ij}}(\tilde{X}_{s}^K) (\gamma_{\mathbf{k}_2}^{\tilde{X}^K}(s))^T-\gamma_{\mathbf{k}_1}(s)M_{\beta_{ij}}(\tilde{X}_{s}^K) (\gamma_{\mathbf{k}_2}(s))^T|,
\end{align*}
where $M_{\beta_{ij}}(x)=\bbE|\beta_{\mathbf{k}_1,i}(x)\cdot(\beta_{\mathbf{k}_2,j}(x))^T|$. Then (AC) and the fact $\tilde{X}_{s}^K$ is integrable imply $\sum_{|\mathbf{k}_1|\leq K}\sum_{|\mathbf{k}_2|\leq K}M_{\beta_{ij}}(\tilde{X}_{s}^K)\lesssim A_\beta^2$. We find
\begin{align*}
|\bbE[R_2(s)]|&\lesssim \sup_{\mathbf{k_1},\mathbf{k_2}}|\gamma_{\mathbf{k}_1}^{\tilde{X}^K}(s)(\gamma_{\mathbf{k}_2}^{\tilde{X}^K}(s))^T-\gamma_{\mathbf{k}_1}(s)(\gamma_{\mathbf{k}_2}(s))^T|\sum_{|\mathbf{k}_1|\leq K}\sum_{|\mathbf{k}_2|\leq K}M_{\beta_{ij}}(\tilde{X}_{s}^K)\\
&\lesssim \sup_{\mathbf{k_1},\mathbf{k_2}}|\gamma_{\mathbf{k}_1}^{\tilde{X}^K}(s)(\gamma_{\mathbf{k}_2}^{\tilde{X}^K}(s))^T-\gamma_{\mathbf{k}_1}(s)(\gamma_{\mathbf{k}_2}(s))^T|.
\end{align*}
By adding and subtracting $1\times1$ matrix $\gamma_{\mathbf{k}_2}^{\tilde{X}^K}(s)(\gamma_{\mathbf{k}_1}(s))^T$, triangle inequality, (AC) and (AF), we obtain
\begin{align}
	& \sup_{\mathbf{k_1},\mathbf{k_2}}|\gamma_{\mathbf{k}_1}^{\tilde{X}^K}(s)(\gamma_{\mathbf{k}_2}^{\tilde{X}^K}(s))^T\mp\gamma_{\mathbf{k}_2}^{\tilde{X}^K}(s)(\gamma_{\mathbf{k}_1}(s))^T-\gamma_{\mathbf{k}_1}(s)(\gamma_{\mathbf{k}_2}(s))^T|\nonumber\\
	=& \sup_{\mathbf{k_1},\mathbf{k_2}}|\gamma_{\mathbf{k}_2}^{\tilde{X}^K}(s)(\gamma_{\mathbf{k}_1}^{\tilde{X}^K}(s)-\gamma_{\mathbf{k}_1}(s))^T+\gamma_{\mathbf{k}_1}(s)(\gamma_{\mathbf{k}_2}^{\tilde{X}^K}(s)-\gamma_{\mathbf{k}_2}(s))^T|\nonumber\\
	&( \text{(AC) means uniform boundedness of}~|\gamma_{\mathbf{k}_2}^{\tilde{X}^K}|~\text{and}~|\gamma_{\mathbf{k}_1}|)\nonumber\\
	\lesssim& \sup_{\mathbf{k_1},\mathbf{k_2}}(|\gamma_{\mathbf{k}_1}^{\tilde{X}^K}(s)-\gamma_{\mathbf{k}_1}(s)|+|\gamma_{\mathbf{k}_2}^{\tilde{X}^K}(s)-\gamma_{\mathbf{k}_2}(s)|)\nonumber\\
	=& \sup_{\mathbf{k_1}}|\gamma_{\mathbf{k}_1}^{\tilde{X}^K}(s)-\gamma_{\mathbf{k}_1}(s)| + \sup_{\mathbf{k_2}}|\gamma_{\mathbf{k}_2}^{\tilde{X}^K}(s)-\gamma_{\mathbf{k}_2}(s)|\nonumber\\
	\lesssim&\sup_{\mathbf{k}}|\gamma_{\mathbf{k}}^{\tilde{X}^K}(s)-\gamma_{\mathbf{k}}(s)|.\nonumber
\end{align}
It means
\begin{align}\label{eq:wkm2}
|\bbE[R_2(s)]|\lesssim\sup_{\mathbf{k}}|\gamma_{\mathbf{k}}^{\tilde{X}^K}(s)-\gamma_{\mathbf{k}}(s)|.
\end{align}

\paragraph*{Error $R_3$:}
\hvdiffs~(bounded second derivatives for $v$) and (AC)  result in 
\begin{align*}
|\bbE[R_3(s)]| &\lesssim \bbE\bigg[\dfrac{1}{2}\smsec|\partial^2_{x_i,x_j} v  (s, \tilde{X}_{s}^K)|\cdot|\sum_{|\mathbf{k}_1|> K}\sum_{|\mathbf{k}_2|\leq K}(\beta_{\mathbf{k}_1,i}(\tilde{X}_{s}^K)\gamma_{\mathbf{k}_1}(s))\cdot(\beta_{\mathbf{k}_2,j}(\tilde{X}_{s}^K)
\gamma_{\mathbf{k}_2}(s))^T|\bigg]\\
&\lesssim \bbE|\sum_{|\mathbf{k}_1|> K}\sum_{|\mathbf{k}_2|\leq K}(\beta_{\mathbf{k}_1,i}(\tilde{X}_{s}^K)\gamma_{\mathbf{k}_1}(s))\cdot(\beta_{\mathbf{k}_2,j}(\tilde{X}_{s}^K)
\gamma_{\mathbf{k}_2}(s))^T|\\
&(\text{non-randomness property of $\gamma_{\mathbf{k}}(s)$})\\
&\lesssim \sum_{|\mathbf{k}_1|> K}\sum_{|\mathbf{k}_2|\leq K}|\gamma_{\mathbf{k}_1}(s)M_{\beta_{ij}}(\tilde{X}_{s}^K) (\gamma_{\mathbf{k}_2}(s))^T|\\
&\lesssim \sum_{|\mathbf{k}_1|> K}|\gamma_{\mathbf{k}_1}(s)|\sup_{\mathbf{k}_2}|\gamma_{\mathbf{k}_2}(s)|\sum_{|\mathbf{k}_1|> K}\sum_{|\mathbf{k}_2|\leq K}M_{\beta_{ij}}(\tilde{X}_{s}^K).
\end{align*}
Then (AC) tells us uniform boundedness of $|\gamma_{\mathbf{k}}|$ and $\sum_{|\mathbf{k}_1|> K}\sum_{|\mathbf{k}_2|\leq K}M_{\beta_{ij}}(\tilde{X}_{s}^K) \lesssim A_\beta^2 $, and we obtain
\begin{align}\label{eq:wkm3}
|\bbE[R_3(s)]|\lesssim  \sum_{|\mathbf{k}|> K}|\gamma_{\mathbf{k}}(s)|.
\end{align}
Similarly, by the condition on the second-order derivatives from \hvdiffs, we can establish that
\begin{align}\label{eq:wkm4}
|\bbE[R_4(s)]|\lesssim \sum_{|\mathbf{k}|> K}|\gamma_{\mathbf{k}}(s)|~\text{and}~|\bbE[R_5(s)]| \lesssim \sum_{|\mathbf{k}|> K}|\gamma_{\mathbf{k}}(s)|.
\end{align}
Combining \eqref{eq:wkm1}, \eqref{eq:wkm2}, \eqref{eq:wkm3}, \eqref{eq:wkm4}, it is clear that
\begin{align}
|\bbE[P(\tilde{X}_{t}^K)]-\bbE[P({X}_{t})]|&=|\int_0^t\bbE[\sum_{i=1}^5R_i(s)]ds|\lesssim \int_0^t\sum_{i=1}^5|\bbE[R_i(s)]|ds\nonumber\\
&\lesssim \int_0^t \bigg(\sup_{\mathbf{k}}|\gamma_{\mathbf{k}}^{\tilde{X}^K}(s)-\gamma_{\mathbf{k}}(s)|+\sum_{|\mathbf{k}|> K}|\gamma_{\mathbf{k}}(s)|\bigg)ds.\label{eq:gP_res}
\end{align}
Next by (AF), we can choose $P$ to be $\varphi_{\mathbf{k}}$ for any $\mathbf{k}\in N^d_0$ and obtain
\begin{align*}
|\bbE[\varphi_{\mathbf{k}}(\tilde{X}_{t}^K)]-\bbE[\varphi_{\mathbf{k}}({X}_{t})]|&=|\int_0^t\bbE[\sum_{i=1}^5R_i(s)]ds|\lesssim \int_0^t\sum_{i=1}^5|\bbE[R_i(s)]|ds\\
&\lesssim \int_0^t \bigg(\sup_{\mathbf{k}}|\gamma_{\mathbf{k}}^{\tilde{X}^K}(s)-\gamma_{\mathbf{k}}(s)|+\sum_{|\mathbf{k}|> K}|\gamma_{\mathbf{k}}(s)|\bigg)ds.
\end{align*}
Then taking supremum over $\mathbf{k}$ for both sides and the Gronwall's lemma result in
\begin{align}\label{eq:phi_res}
\sup_{\mathbf{k}}|\bbE[\varphi_{\mathbf{k}}(\tilde{X}_{t}^K)]-\bbE[\varphi_{\mathbf{k}}({X}_{t})]|\lesssim \sum_{|\mathbf{k}|> K}\sup_{0< t\leq T}|\gamma_{\mathbf{k}}(t)|,\quad \forall t\in(0,T].
\end{align}
Thus, combining \eqref{eq:gP_res} and \eqref{eq:phi_res} completes the proof.
\end{proof}

\section{Main result of the iterative MLMC algorithm} \label{eq:mrimlmc}

We combine the MLMC and Picard iteration to improve the computational efficiency. Fix $m$ and $L$. Take the same family of partitions $\Pi^\ell=\{0=t_{0}^{\ell},\ldots,t^{\ell}_k,\ldots, T=t_{2^{\ell}}^\ell \} $, $\ell=0,\ldots,L$, with $t_k^\ell-t_{k-1}^\ell=h_\ell$ and define $\eta_\ell{(t)}\defeq t_k^\ell\, \text{if}~t\in[t_k^\ell,t_{k+1}^\ell)$. In this case, we mainly approximate \eqref{eq:defxtminteract_project_K}  by $\{\overline{X}_t^{m,K}\}_{t\in[0,T]}$  defined as
\begin{gather}\label{eq:xminnonteract}
	d\overline{X}^{m,K}_t = \sum_{|\mathbf{k}|\leq K}\alpha
_{\mathbf{k}}(\overline{X}^{m,K}_t)\gamma^{\overline{X}^{m-1,K}}_{\mathbf{k}}(t)\, dt + \sum_{|\mathbf{k}|\leq K}\beta
_{\mathbf{k}}(\overline{X}^{m,K}_t)\gamma^{\overline{X}^{m-1,K}}_{\mathbf{k}}(t)\, dW_t\,.
\end{gather}
where 
$
	\gamma^{\overline{X}^{m-1,K}}_{\mathbf{k}}(t) := \bbE
\left[  \varphi_{\mathbf{k}}(\overline{X}_{t}^{m-1,K})\right].
$
To simulate \eqref{eq:xminnonteract}, we only need to approximate the relevant expectations, without dependence on $x$, with respect to the law of the process at the previous Picard step $m-1$ and the time grid $\Pi^{L}$, i.e.   
\[
\begin{split}
	 \left(\mathbb{E}[\varphi_{\mathbf{k}}(\overline {X}^{m-1,K}_0)],\ldots, \mathbb{E}[\varphi_{\mathbf{k}}(\overline{X}^{m-1,K}_{t_k^{L}})]\ldots,  \mathbb{E}[\varphi_{\mathbf{k}}(\overline{X}^{m-1,K}_{T})]\right),~|\mathbf{k}|\leq K.
\end{split}
\]
By approximating these expectations with the MLMC samples $\mathcal{M}^{(m-1,K)}_{\cdot,t}$, the Euler type \textit{iterative MLMC particle} method relies on the system of SDEs:
\begin{equation} \label{eq:contiousYinteractnon}
 dY^{i,K, m,\ell}_t =  \sum_{|\mathbf{k}|\leq K}\alpha
_{\mathbf{k}}(Y^{i,K, m,\ell}_{\eta_{\ell} (t)})\mathcal{M}^{(m-1,K)}_{\mathbf{k},\eta_{\ell} (t)}  \, dt +\sum_{|\mathbf{k}|\leq K}\beta
_{\mathbf{k}}(Y^{i,K, m,\ell}_{\eta_{\ell} (t)})\mathcal{M}^{(m-1,K)}_{\mathbf{k},\eta_{\ell} (t)} \, dW^{i,m}_t\,, 
\end{equation}
where $Y^{i,0, \ell}= X_0$ and $\mathcal{M}^{(m-1,K)}_{\mathbf{k},t}$ is defined as
\begin{equation}\label{eq:nonM1}
	\mathcal{M}^{(m-1,K)}_{\mathbf{k},t} = \sum_{\ell=0}^L\dfrac{1}{N_{m-1,\ell}}\sum_{i=1}^{N_{m-1,\ell}}\varphi_{\mathbf{k}}(Y_t^{i,K,m-1,\ell})-\varphi_{\mathbf{k}}(Y_t^{i,K,m-1,\ell-1}).
\end{equation}
where $\varphi_{\mathbf{k}}(Y^{\cdot,\cdot,\cdot,-1})=0$. Under the assumptions listed in Section \ref{sec:desmlmc}, we derive the precise error bound for \eqref{eq:contiousYinteractnon} as confirmed in Theorem \ref{lm:mseboundfinalinteract} by following similar lines as presented in \cite{szpruch2019}.
\begin{theorem}\label{lm:mseboundfinalinteract}
Assume \textbf{(Ker-Reg)}{ } \, and \textbf{($\mu_0$-$L_p$)}{ } .  Fix $M>0$ and let $P \in C^2_b (\mathbb{R}^d,\mathbb{R}^d)$.  Define
$
MSE_t^{(M,K)}(P) := \mathbb{E}[( \mathcal{M}^{(M,K)}_{t}(P) - \mathbb{E}[P(\tilde{X}^K_{t})] )^2].
$
Then there exists a constant $c>0$ (independent of the choices of $M$, $L$ and $\{N_{m,\ell}\}_{m,\ell}$) such that for every $t \in [0,T]$,
\[
	MSE_{\eta_L(t)}^{(M,K)} (P)\leq c\bigg\{h_L^2+\sum_{m=1}^{M}\frac{c^{M-m}}{(M-m)!}\cdot\sum_{\ell=0}^L\frac{h_\ell}{N_{m,\ell}} +\frac{c^{M-1}}{M!}\bigg\},
\]
\end{theorem}
where $\calM_{t}^{(M,K)}(P)$ is given by
\begin{equation}\label{eq:convergeresnon}
	\calM_{t}^{(M,K)}(P) = \sum_{\ell=0}^L\dfrac{1}{N_{M,\ell}}\sum_{i=1}^{N_{M,\ell}}P(Y_t^{i,K,M,\ell})-P(Y_t^{i,K,M,\ell-1}),\quad P(Y_t^{\cdot,\cdot,\cdot,-1})=0.
\end{equation}

\begin{remark}
If one use iterative MLMC containing $M$ Picard steps to reach the required MSE accuracy level $\epsilon^2$, the cost of the first Picard step is $\sum_{\ell=0}^{L} h^{-1}_{\ell} N_{0, \ell}$ and that of rest Picard steps is $\sum_{m=2}^M\bigg(\sum_{\ell=0}^{L} h^{-1}_{\ell} N_{m, \ell} +K\cdot\sum_{\ell'=0}^{L}  h^{-1}_{\ell'} N_{m-1, \ell'} \bigg)$. Thus the sum of these cost gives the total cost.
\end{remark}
\subsection{Algorithm}
The complete MLMC particle approximation via projection is presented in a schematic form as Algorithm \ref{Algorithm}.

\begin{algorithm*}[ht]
\KwIn{Initial measure $\mu^0$ for $X^0$, MSE level $\epsilon$, payoff function $P(\cdot):\mathbb{R}^d\rightarrow\mathbb{R}$ and $M$, the number of Picard steps.}
\KwOut{$\mathcal{M}_{T}^M(P)$, the approximation for the quantity $\mathbb{E}[P(X_T)]$.}
\nl Find $K(\epsilon)$ according to \eqref{eq:Kbound}\;
\nl Fix $L_0$. Apply standard MLMC to the process $X^0$ to calculate the approximation matrix of size $K\times(1+2^{L_0})$
\[
\bar{D}^0:=
 \begin{pmatrix}
  \mathcal{M}_{1,t_0}^{(0)}&,\ldots,&\mathcal{M}_{1,t_{2^{L_{0}}}}^{(0)}
 \end{pmatrix}.
 \]
\nl \For{ $m = 1$ to $M-1$}{
	\nl Fix $L_m$, the largest level in Multilevel in $mth$ Picard step\; 
	\nl Conditioning on $\bar{D}^{m-1}$ (starting from $\bar{D}^0$), take \eqref{eq:contiousYinteractnon} and run MLMC to obtain the matrix of of size $K\times(1+2^{L_0})$ :
	\vspace{-0.2cm}
	\[
	\bar{D}^m:=
 \begin{pmatrix}
  \mathcal{M}_{1,t_0}^{(m)}&,\ldots,&\mathcal{M}_{1,t_{2^{L_{0}}}}^{(m)}
 \end{pmatrix}.
 \]}
\nl Fix $L_M$, the largest level in MLMC at $Mth$ Picard step\;
\nl Conditioning on $\bar{D}^{M-1}$, run standard MLMC with interpolation to obtain the final vectors of approximations  $(\mathcal{M}_{t_0}^M(P),\ldots,\mathcal{M}_{t_{2^{L_M}}}^M(P))$\;
\nl {\bf Return} $\mathcal{M}_{t_{2^{L_M}}}(P)$, i.e. $\mathcal{M}_{T}(P)$.
    \caption{{\bf Iterative MLMC via projections} \label{Algorithm}}

\end{algorithm*}

\section{Complexity}
\textbf{Approximation of Expectations}
\ First, we present the complexity theorem for projections w.r.t. the basis \eqref{eq:herm_func}, that is, we consider the particle system \eqref{eq:defxtminteract_project_Keuler}. 
\begin{theorem}\label{thm:c2nonintK}
Let $P \in C^{2}_{b}( \mathbb{R}^d, \mathbb{R}^d)$ be a Lipschitz continuous function.
Assume (AC), (AF) and (AD) hold. Then there exists a constant $c$ such that for any $\epsilon<e^{-1}$ there exists $K$, $L$ and $N$ such that 
\[
	MSE \vcentcolon=  \sup_{0\leq t\leq T}\bbE\bigg[\Bigl(\frac{1}{N}\sum_{i=1}^NP(\bar{X}^{i,N,K}_{\eta_L(t)}) - \bbE[P(X_{\eta_L(t)})]\Bigr)^2\bigg]< \epsilon^2,
\]
with the corresponding computational cost  being of order $\epsilon^{-3}\log(\epsilon^{-1})$.
\end{theorem}
\begin{proof}
From Theorem \ref{lm:wkmain}, (AD), the property of geometric series and standard results for the underlying particle system, we obtain
\begin{align*}
	\sup_{0\leq t\leq T}\bbE\bigg[\Bigl(\frac{1}{N}\sum_{i=1}^NP(\bar{X}^{i,N,K}_{\eta_L(t)}) - \bbE[P(X_{\eta_L(t)})]\Bigr)^2\bigg]&\lesssim\bigg\{(\sum_{\mathbf{k}>K} |\gamma_{\mathbf{k}}|)^2+h_L^2+\frac{1}{N}\bigg\}\\
	&\lesssim\bigg\{\exp(-2\gamma^\circ K)+h_L^2+\frac{1}{N}\bigg\}.
\end{align*}
It induces the following optimisation problem
 \begin{align}
	&\min_{K,L, N}C(K,L,N)=N\cdot K\cdot h_L^{-1} \label{eq:constraintmain1K}\\&\text{s.t.}~ 	\exp(-2\gamma^\circ K)+h_L^2+\frac{1}{N}\lesssim\epsilon^2\label{eq:constraintcom1K}.
\end{align}
Then \eqref{eq:constraintcom1K} leads to $K\lesssim \log(\epsilon^{-1})$, $h_L\lesssim\epsilon$ and $N\lesssim \epsilon^{-2}$. Combining these and \eqref{eq:constraintmain1K} completes the proof.
\end{proof}
Next, we present the complexity theorem for iterative MLMC with projected coefficients approch for the particle system \eqref{eq:contiousYinteractnon}.
\begin{theorem}\label{thm:c2nonint}
 Let $P \in C^{2}_{b}( \mathbb{R}^d, \mathbb{R}^d)$ be a Lipschitz continuous function and $Y^{\cdot, K, m,\ell}$ be defined in \eqref{eq:contiousYinteractnon}. Assume that \textbf{(Ker-Reg)}{ }, \textbf{($\mu_0$-$L_p$)}{ }, (AC), (AF) and (AD) hold.
 Then there exists a constant $c$ such that for any $\epsilon<e^{-1}$ there exists $M$, $K$, $L$ and $\{N_\ell\}_{\ell=0}^L$ such that
\[
	MSE^{(M,K)} \vcentcolon=  \sup_{0\leq t\leq T}\bbE\bigg[(\calM_{\eta_L(t)}^{(M,K)}(P) - \bbE[P(X_{\eta_L(t)})])^2\bigg]< \epsilon^2,
\]
where $\calM_{t}^{(M,K)}(P)$ is given by \eqref{eq:convergeresnon}, the corresponding computational cost  being of order $\epsilon^{-2}|\log\epsilon|^4$. \end{theorem}
\begin{proof}
With (AC) and (AF), Theorem \ref{lm:wkmain} and (AD) imply
\begin{align*}
\sup_{0\leq t\leq T}\bbE\bigg[(\calM_{\eta_L(t)}^{(M,K)}(P) &- \bbE[P(X_{\eta_L(t)})])^2\bigg] \lesssim \bigg\{|\mathbb{E}[P(\tilde{X}^K_{\eta_L(t)})]-\bbE[P(X_{\eta_L(t)})]|^2|\\
&+\mathbb{E}|\calM_{\eta_L(t)}^{(M,K)}(P)-\mathbb{E}[P(\tilde{X}^K_{\eta_L(t)})]|^2\bigg\}\\
& \lesssim (\sum_{\mathbf{k}>K} |\gamma_{\mathbf{k}}|)^2+\mathbb{E}|\calM_{\eta_L(t)}^{(M,K)}(P)-\mathbb{E}[P(\tilde{X}^K_u)]|^2\\
&\lesssim \exp(-2\gamma^\circ K) + \mathbb{E}|\calM_{\eta_L(t)}^{(M,K)}(P)-\mathbb{E}[P(\tilde{X}^K_u)]|^2.
\end{align*}
By \textbf{(Ker-Reg)}{ } and \textbf{($\mu_0$-$L_p$)}{ }, Theorem \ref{lm:mseboundfinalinteract} further gives
\begin{align*}
\sup_{0\leq t\leq T}\bbE\bigg[(\calM_{\eta_L(t)}^{(M,K)}(P) &- \bbE[P({X}_{\eta_L(t)})])^2\bigg] \lesssim \bigg\{\exp(-2\gamma^\circ K) +h_L^2+\sum_{\ell=0}^L\frac{h_\ell}{N_\ell}\\
&+\int_{0}^{T}\bbE|[\calM_{\cdot,\eta_L(t)}^{(M-1,K)}(\varphi_\cdot)- \bbE[\varphi_{\cdot}(\tilde{X}_{\eta_L(t)}^K)]|^2dt\bigg\}\\
& \lesssim\cdots\lesssim\bigg\{\exp(-2\gamma^\circ K) +h_L^2+\sum_{\ell=0}^L\frac{h_\ell}{N_{\ell}}+\frac{1}{M!}\bigg\},
\end{align*}
where $N_{\ell} = \max_m\{N_{m,\ell}\}$. It induces the following optimisation problem  
\begin{align}
	&\min_{M,K,L, \{N_\ell\}_{\ell=0}^L}C(M,K,L,N_0,\ldots,N_L)=M\cdot K\cdot\sum_{\ell=0}^L h_\ell^{-1} N_\ell\label{eq:constraintmain1}\\&\text{s.t.}~ 	\exp(-2\gamma^\circ K) +h_L^2+\sum_{\ell=0}^L\frac{h_\ell}{N_{\ell}}+\frac{1}{M!}\lesssim\epsilon^2\label{eq:constraintcom1}.
\end{align}
In a similar way as in Giles's complexity theorem in \cite{giles2015multilevel}, we see that the only difference is presence of multiplicative constants $M$ and $K$  in \eqref{eq:constraintmain1} and additive cost functions $\exp(-2\gamma^\circ K) $ and $\frac{1}{M!}$ in \eqref{eq:constraintcom1}. First, we have 
\begin{align}\label{eq:Kbound}
\exp(-2\gamma^\circ K) \lesssim \epsilon^2 \implies K\lesssim  \log(\epsilon^{-1}).
\end{align} 
By the Stirling approximation in \cite{10.2307/2308012}, given by $\sqrt{2\pi}n^{n+\frac{1}{2}}e^{-n}\leq n!, \forall n\geq 1,$ , along with \eqref{eq:Kbound}, we obtain
\begin{align}\label{eq:Mestimate}
	\frac{1}{M!}\lesssim \frac{1}{\sqrt{2\pi}(\frac{M}{e})^M}\lesssim \epsilon^2\implies  M\lesssim\log(\epsilon^{-1}).
\end{align}
From \cite{giles2008multilevel}, we know that $\sum_{\ell=0}^{L}N_\ell h_\ell^{-1}\lesssim \epsilon^{-2}|\log\epsilon|^2$. Combining this, \eqref{eq:constraintmain1}, \eqref{eq:Kbound} and \eqref{eq:Mestimate} yields the result.

\end{proof}
%

\textbf{Density estimation} Let us now discuss the estimation of the densities $\mu_{t},$ $t\geq0.$ Fix
some $t>0,$ $K\in \mathbb{N}^d$ and set
\[
\widehat{\mu}_{t}^{K}(y):=\sum_{|\mathbf{k}|\leq K}\gamma_{\mathbf{k}}^{A}(t)\varphi_{\mathbf{k}}(y)w(y),
\]
where $\gamma_{\mathbf{k}}^{A}(t)$ is a type of approximation for $\gamma_{\mathbf{k}}$ such as $\hat{\gamma}_{\mathbf{k}}$ in \eqref{eq:defxtminteract_project_Keuler} and $\mathcal{M}_{\mathbf{k},t}^{(M)}$ in \eqref{eq:contiousYinteractnon}.
We again work with MSE in the similar form of \eqref{eq:msegeneral} and we obtain
\begin{align}\label{eq:dsstart}
\widehat{MSE}&\defeq\sup_{0< t\leq T}\mathbb{E}\int|\widehat{\mu}_{t}^{K}(y)-\mu_{t}(y)|^{2}w^{-1}(y)\,dy\nonumber\\
&=\sum_{|\mathbf{\mathbf{k}}|\leq K}\sup_{0< t\leq T}\mathbb{E}\left[  |\gamma_{\mathbf{k}}^{A}(t)-\gamma_{\mathbf{k}}(t)|^{2}\right]
+\sum_{|\mathbf{k}|>K}\sup_{0< t\leq T}|\gamma_{\mathbf{k}}(t)|^{2},
\end{align}
If we assume that the coefficients \(b\) and \(\sigma\) are infinitely smooth w.r.t the basis \eqref{eq:herm_func} in the sense that  
\begin{eqnarray*}
A_{\mathbf{k},\alpha}\lesssim \exp(-\alpha^\circ|\mathbf{k}|), \quad A_{\mathbf{k},\beta}\lesssim \exp(-\beta^\circ|\mathbf{k}|),\quad |\mathbf{k}|\to \infty
\end{eqnarray*}
for some constants \(\alpha^\circ,\beta^\circ>0,\) then it can be shown along the same lines as in the proof of Theorem~2.2 in \cite{antonelli2002rate} that 
\begin{eqnarray}\label{eq:gammadecay}
|\gamma_{\mathbf{k}}(t)|\lesssim \exp(-\gamma^\circ|\mathbf{k}|), \quad |\mathbf{k}|\to \infty
\end{eqnarray}
for some \(\gamma^\circ>0,\) provided that the diffusion coefficient satisfies the H\"ormander 
condition. As shown in Example~\ref{exm:hermite}, \(D_{\mathbf{k},\varphi}\) and \(L_{\mathbf{k},\varphi}\) are uniformly bounded in \(\mathbf{k}.\) To ensure $MSE<\epsilon_0^2$, combing \eqref{eq:dsstart} and \eqref{eq:gammadecay}, together with the property of geometric series implies $K\lesssim |\log(\epsilon_0)|$. Therefore, for any $|\mathbf{k}|<K$, we need
\[
	\sup_{0\leq t\leq T}\mathbb{E}\left[  |\gamma_{\mathbf{k}}^{A}(t)-\gamma_{\mathbf{k}}(t)|^{2}\right]= \sup_{0\leq t\leq T}\mathbb{E}\left[  |\gamma_{\mathbf{k}}^{A}(t)-\mathbb{E}[\varphi(X_t)]|^{2}\right]	\lesssim \dfrac{\epsilon_0^2}{K}\lesssim  \dfrac{\epsilon_0^2}{|\log(\epsilon_0)|}
\]
It means if we set the accuracy level $\epsilon$ as $\frac{\epsilon_0}{\sqrt{|\log(\epsilon_0)|}}$ in previous complexity theorems for approximation of expectations,
 we find the optimal parameter setting corresponding to $\gamma_{\mathbf{k}}^{A}(t)$ for both methods. By computation, the complexities are in the magnitude of 
 \[
 \mathcal{O}\left(\epsilon_0^{-3}{|\log(\epsilon_0)|^{5/2}}\right)
 \]
 and
 \[
 \mathcal{O}\left({\epsilon_0}^{-2}|\log\epsilon_0|^5\right),
 \] 
 respectively. 
\section{Numerical Experiments}
\paragraph*{\textbf{Projection particle method}}
Consider the MVSDE of the form: 
\begin{eqnarray}\label{eq:oriMVSDE}
dX_t=\bbE_{X'}\left[Q(X_t-X'_t)\right]\,dt+\sigma\,dW_t,\quad t\in [0,1],\quad X_0 = 0.5,
\end{eqnarray}
i.e. of the form \eqref{eq:mlmceulerlineartargetinteract} with \(\bar{b}(x,y)=Q(x-y),\) \(\bar{\sigma}(x,y)=\sigma\) and \(\mu_0(x)=(1/\sqrt{2\pi}) e^{-x^2/2}.\) Let us use the Hermite basis (see section \ref{sec:Projections}) to approximate the density of \(X_t\) for any \(t\in [0,1].\) In the case \(Q(x)=e^{-x^2/2},\) we explicitly derive
\begin{eqnarray*}
\int_{\mathbb{R}} e^{-(x-y)^{2}/2-x^{2}/2}H_{n}(x)\,dx&=&\frac{e^{-y^{2}/4}}{2}\int e^{-(z-y)^{2}/4}H_{n}(z/2)\,dz
\\
&=& \sqrt{\pi}\,\frac{e^{-y^{2}/4}}{2}\left(\frac{1}{2}\right)^{n-1}(2y)^{n}.
\end{eqnarray*}
As a result
\begin{eqnarray*}
\alpha_n(y)=\int e^{-(x-y)^{2}/2-x^{2}/2}\overline{H}_{n}(x)\,dx=\pi^{1/4}\left(\frac{1}{2}\right)^{n/2}\frac{y^{n}}{\sqrt{n!}}e^{-y^{2}/4},
\end{eqnarray*}
where \(\overline{H}_{n}\) stands for the normalised Hernite polynomial of order \(n.\) We take \(\sigma=0.1.\)  Using the Euler scheme with time step \(h=1/L=0.01\),  we first simulate \(N=500\) paths of the time discretised  process \(\bar{X}^{\cdot,N}.
\)  Next,   by means of the closed form expressions for \(\alpha_n,\)   we generate \(N\) paths of the projected approximating process  \(\bar{X}^{\cdot,K,N},\) \(K\in \{1,\ldots,20\}\) using the same Wiener increments as for \(\bar{X}^{\cdot,N},\) so that the approximations \(\bar{X}^{\cdot,N}\)  and \(\bar{X}^{\cdot,K,N}\) are coupled. Finally, we compute the strong approximation error 
\begin{eqnarray*}
E_{N,K}=\sqrt{\frac{1}{N}\sum_{i=1}^{N} \bigl(\bar{X}_{1}^{i,K,N}-\bar{X}_{1}^{i,N}\bigr)^2 }
\end{eqnarray*}
and record times needed to compute approximations \(\bar{X}_{1}^{\cdot,N}\) and \(\bar{X}_{1}^{\cdot,K,N},\) respectively. Figure~\ref{fig:ppm} shows the logarithm of   \(E_{N,K}\) versus the logarithm of the corresponding computational time differences for values \(K\in \{1,\ldots,20\}\). As can be seen, the relation between  logarithmic strong error and  logarithmic computational time gain can be well approximated by a linear function. On the left-hand side of Figure~\ref{fig:ppm} we depict the projection estimate for the density of \(X_1\) corresponding to \(K=10.\) 

\begin{figure}[!htbp]
\centering
\begin{subfigure}{.49\textwidth}
  \centering
  \includegraphics[width=.9\linewidth]{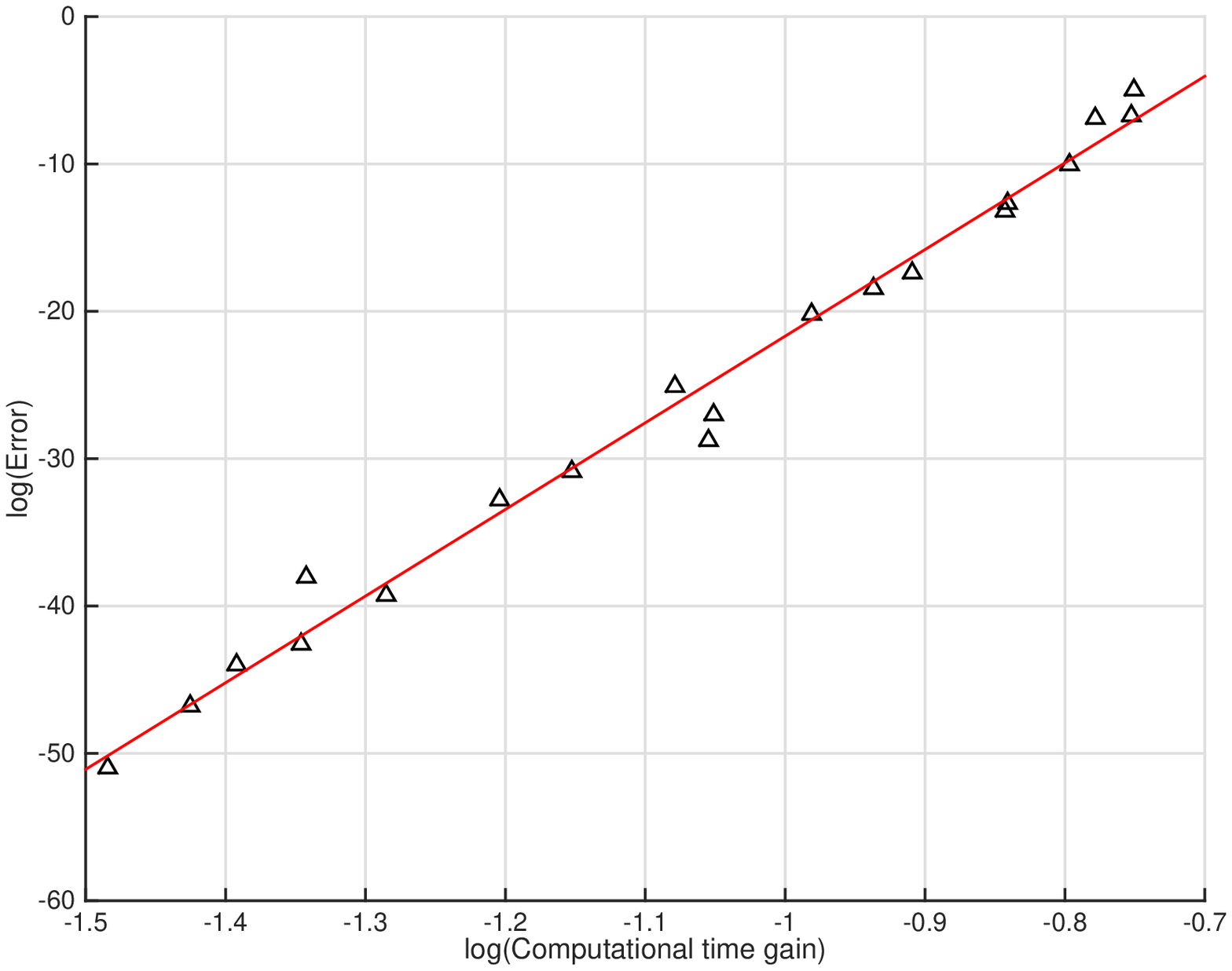}
   \caption{Strong Error $E_{500,K}$ between PPM and Chaos}
  \label{fig:ppm1}
\end{subfigure}
\centering
\begin{subfigure}{.49\textwidth}
  \centering
  \includegraphics[width=.9\linewidth]{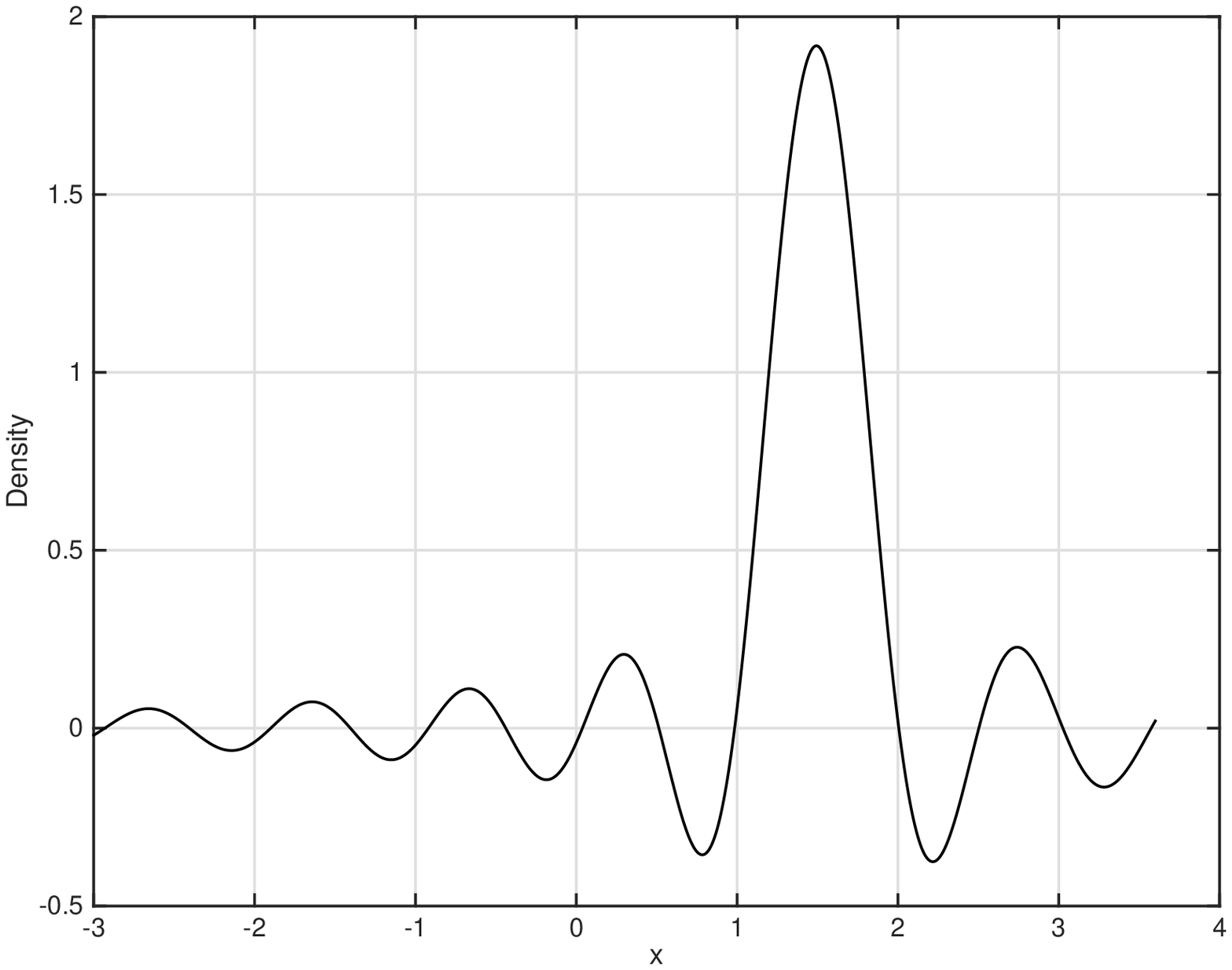}
\caption{Estimated density of $X_1$ under $K=10$}
  \label{fig:ppm2}
\end{subfigure}
\caption{Left: Strong error $E_{500,K}$ between the solution of projected (see \eqref{eq:defxtminteract_project_Keuler}) and non-projected time-discretised particle systems (see \eqref{eq:eulerpartcle}) versus the difference (gain) in computational time. Right: Estimated density of $X_1$ using 11 basis functions.}
\label{fig:ppm}
\end{figure}
\begin{remark}
Note that in Figure \ref{fig:ppm} (also Figure \ref{fig:denest}) the approximated density takes negative values but for example it can be resolved by a certain shifting transformation (see section 2.7 in \protect\cite{belomestny2005implementing}).
\end{remark}

\paragraph*{\textbf{Iterated MLMC on $Kth$ order Hermite Projected SDE}}
Likely, consider one-dimensional $Kth$ order Hermite Projected SDE \eqref{eq:Kres}:
 \begin{align}\label{eq:Kres}
 dX_t = \sum_{k=0}^K\bbE[\varphi_k(X_t)] \alpha_k(X_t)dt +\sigma dW_t, \quad t\in [0,1],\quad X_0=0.5,
 \end{align}
 where $T =1, \sigma = 0.1$, $\varphi_k(x) = \bar{H}_k(x)e^{-\frac{x^2}{2}}$ and $\alpha_k(x)=\pi^{1/4}\left(\frac{1}{2}\right)^{k/2}\frac{x^{k}}{\sqrt{n!}}e^{-x^{2}/4}$.
 \paragraph*{\textbf{Tests of convergence rates}}
As to initialisation,  let $X^0 \sim N(0.5,1)$, $X_0=0.5$ and $L=5$ (recall that $h_\ell=T\cdot 2^{-\ell}$). In the case when $X^0 \sim N(0.5,1)$ (does no change with time) first Picard $X^{1}$ 
is just an SDE with drift at each time-step being exactly the same. We perform tests across all Picard steps and obtain convergence rates are stable during all Picard iterations as depicted on Figure \ref{fig:mlmc_weak}.
Results are shown in Figure \ref{fig:mlmc_weak}.
\begin{figure}[!htbp]

\begin{subfigure}{.49\textwidth}
  \centering
  \includegraphics[width=.9\linewidth]{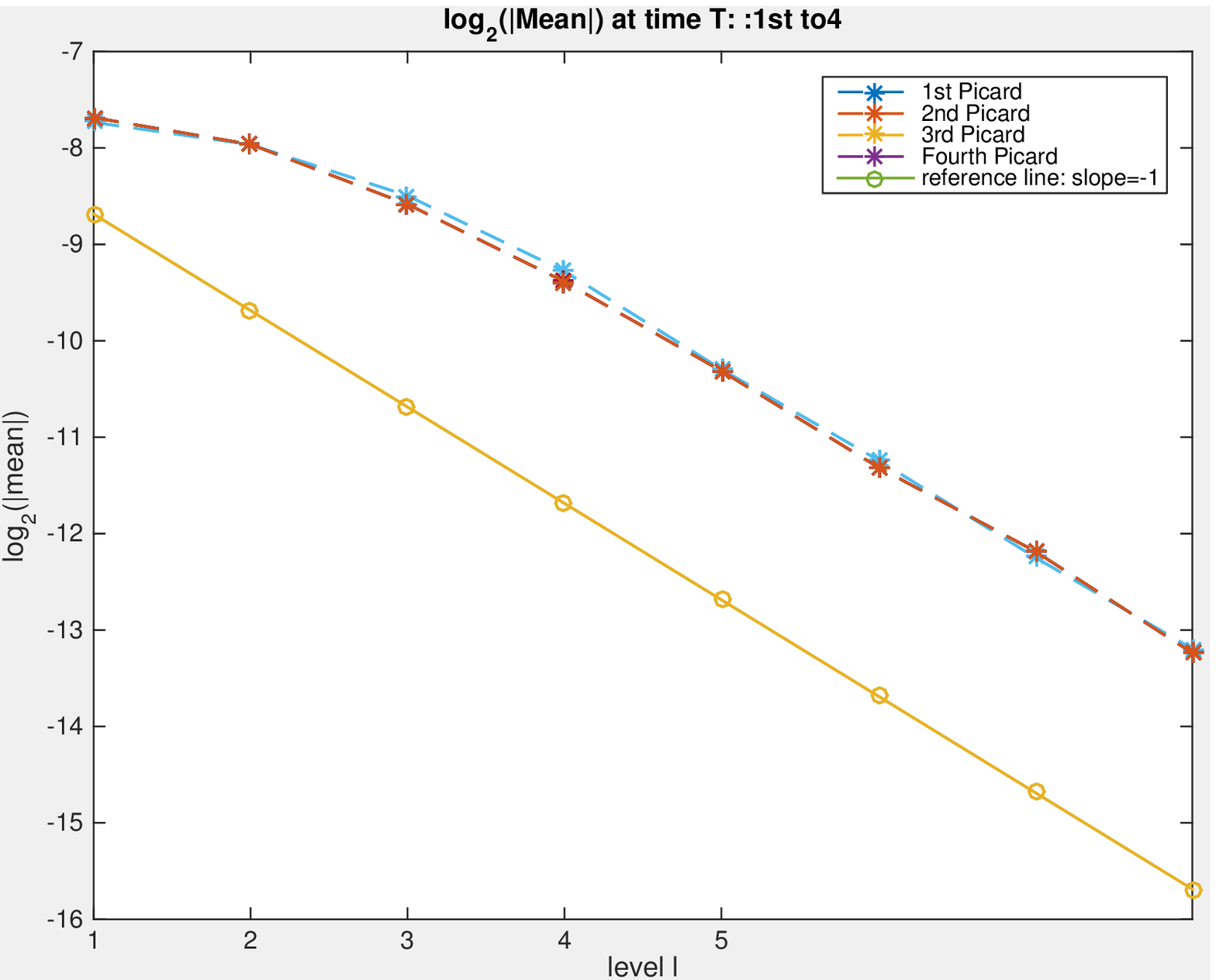}
  \caption{weak rate}
  \label{fig:noninteractingweak}
\end{subfigure}
\begin{subfigure}{.49\textwidth}
  \centering
  \includegraphics[width=.9\linewidth]{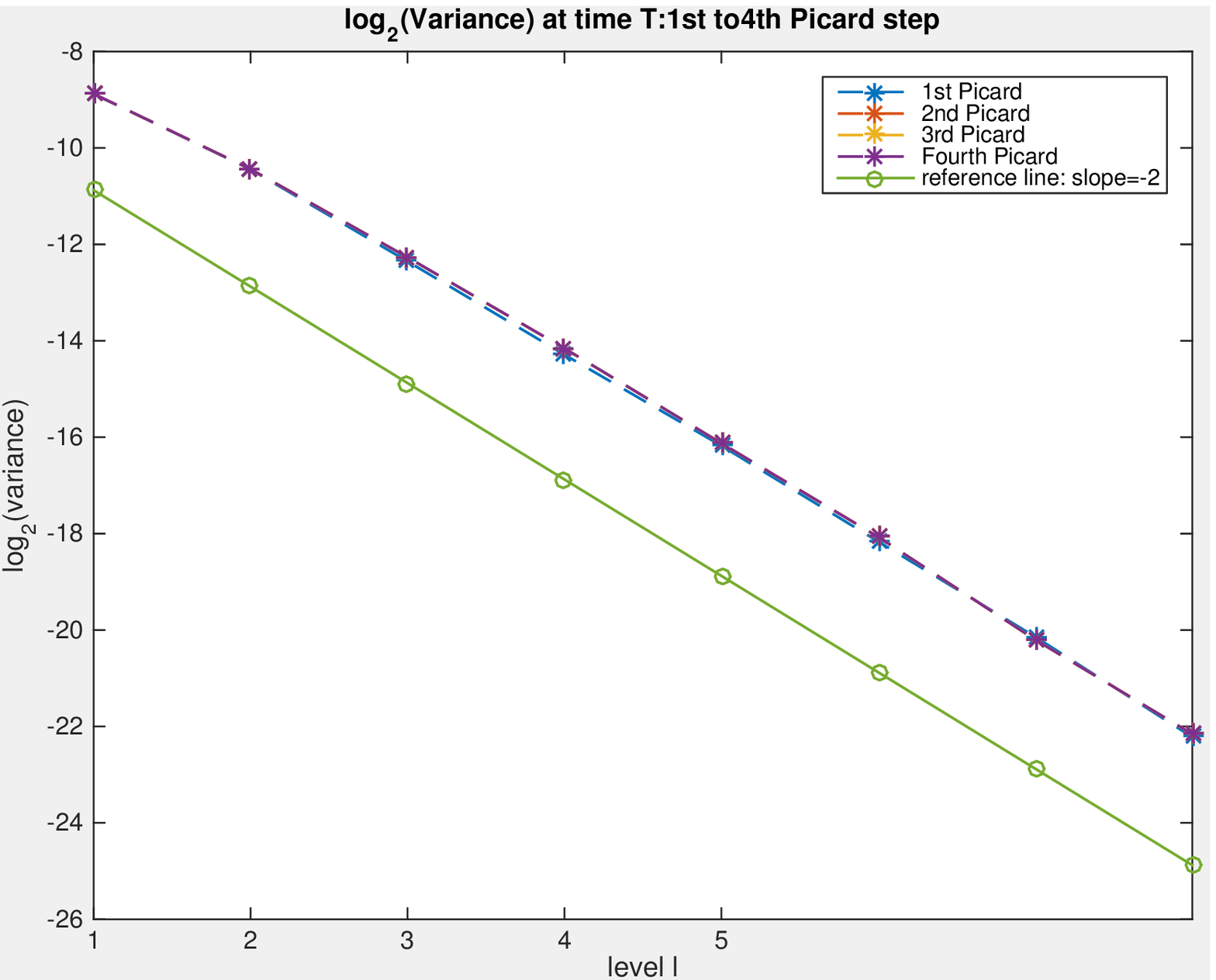}
  \caption{strong rate}
  \label{fig:noninteractingstrong}
  \label{fig:mlmc_weake}
\end{subfigure}%
\caption{In those tests, $N_m$ is $100000$ for all Picard steps and the terminal $T=1$. The sequence $(a_\ell)$ corresponding to weak error of $\varphi_0$ turns into
$
	a_\ell\vcentcolon=\bigg|\frac{1}{N}\sum_{i=1}^{N}\varphi_0(Y_t^{i,m,\ell})-\varphi_0(Y_t^{i,m,\ell-1})\bigg|.
$
Similarly, the sequence $(b_\ell)$ corresponding to strong error of $\varphi_0$ becomes
$
b_\ell\vcentcolon= \frac{1}{N}\sum_{i=1}^{N}\bigg|\varphi_0(Y_t^{i,m,\ell})-\varphi_0(Y_t^{i,m,\ell-1})\bigg|^2.
$ The reference lines in Figure \ref{fig:noninteractingweak} and Figure \ref{fig:noninteractingstrong} indicates $\alpha=1$ and $\beta = 2$ respectively.}
\label{fig:mlmc_weak}
\end{figure}



\textbf{Comparison of all three methods.}
In this section, we compare complexities among standard Particle system applied to the original MVSDE \eqref{eq:oriMVSDE}, Projected Particle system applied to SDE \eqref{eq:Kres} and iterated MLMC (MLMC with Picard) with projected coefficients applied to SDE \eqref{eq:Kres}. The measurement of computational cost  consists of  the number of random gerenations, evaluations of the drift coefficient while ignoring a small amount of constant computational cost induced by linear interpolation operations.
\begin{remark}
We consider the particle system with $5\cdot10^5$ particles and $2^{10}$ timesteps as proxy for the explicit solution of equation \eqref{eq:oriMVSDE} and it corresponds to the benchmark value 1.4951. The MSEs then have been verified, for example, if we require MSEs to be  less than $\epsilon=0.03$, the values lie in the desired region (1.4951 $\pm$ $2\epsilon$). 
\end{remark}

\textbf{Complexity.}
\begin{figure}[!htbp]

\centering
\begin{subfigure}{.49\textwidth}
  \centering
  \includegraphics[width=.9\linewidth]{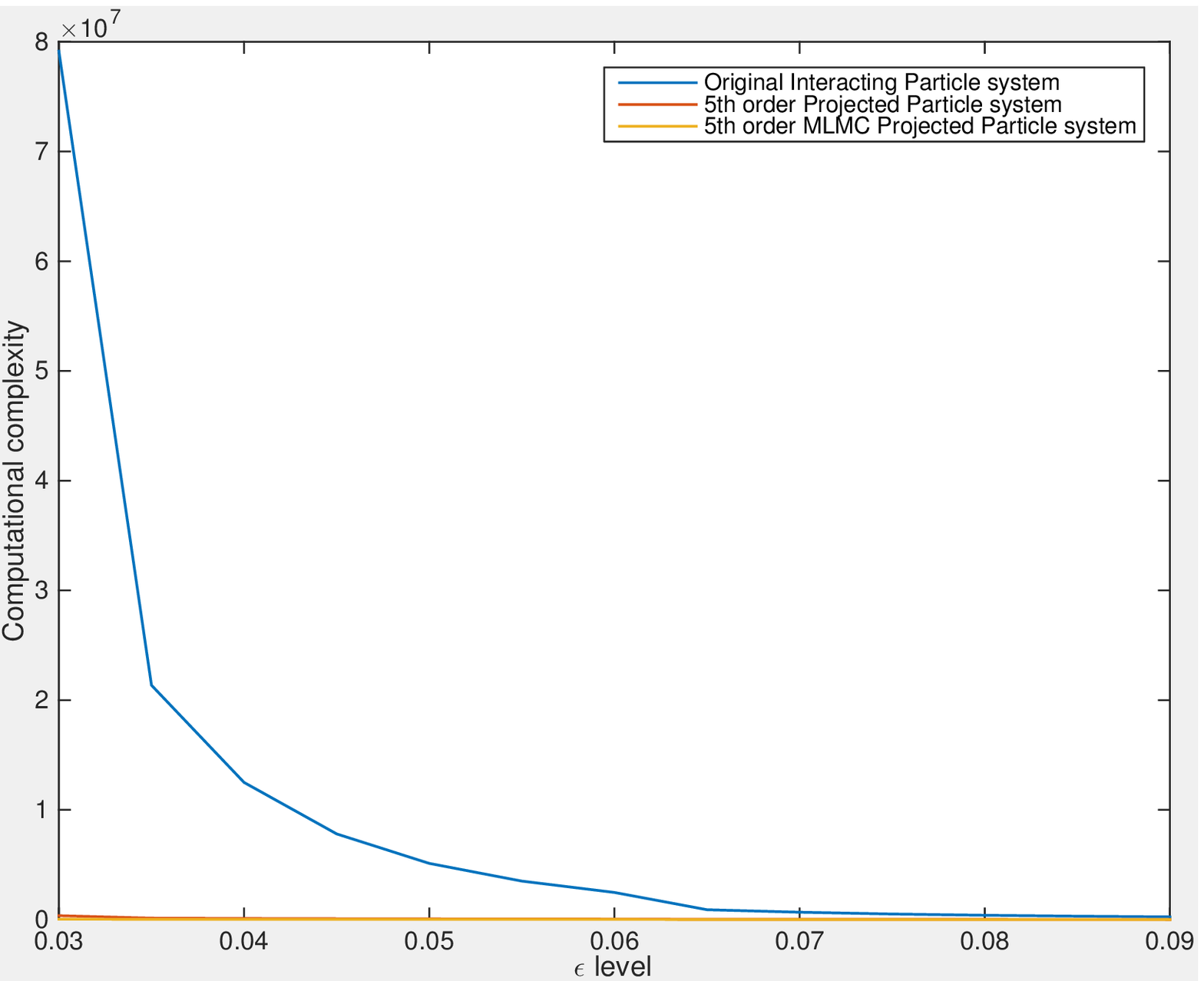}
  \caption{Chaos vs PPM vs Iterated MLMC with Projected coefficients}
  \label{fig:complexity1}
\end{subfigure}
\centering
\begin{subfigure}{.49\textwidth}
  \centering
  \includegraphics[width=.9\linewidth]{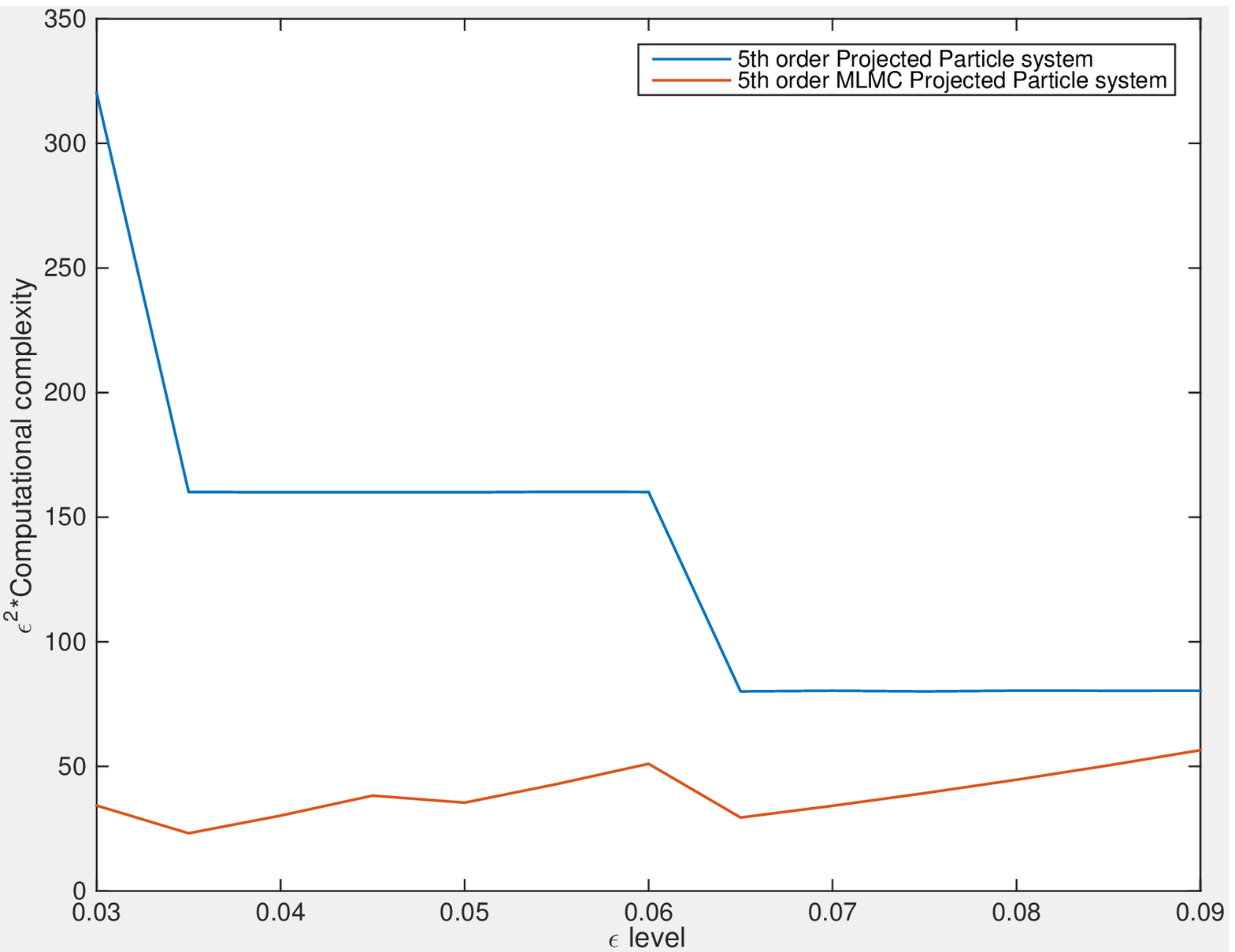}
  \caption{PPM vs Iterated MLMC with Projected coefficients}
  \label{fig:complexity2}
\end{subfigure}
\caption{Chaos stands for standard particle system method applied to the original interacting MVSDE; PPM means Projected Particle method;  From Figure \ref{fig:complexity1} , we see standard particle system has much higher order than $\epsilon^{-2}$ and also observe from Figure \ref{fig:complexity2} that MLMC with projected coefficients is close to $O(\epsilon^{-2})$ while projections close to $O(\epsilon^{-3})$ which coincides with the Theorems \ref{thm:c2nonintK} and \ref{thm:c2nonint} for complexity.}
\label{fig:complexity}
\end{figure}
Figure \ref{fig:complexity} presents the comparison of those complexities. Figure \ref{fig:complexity1} shows standard Particle system applied to the original MVSDE \eqref{eq:oriMVSDE} is the worst one. Figure \ref{fig:complexity2}  indicates that iterated MLMC with projected coefficients can be applied to SDE \eqref{eq:Kres} to further improve the Projected Particle method after reducing interactions from $N$ to $K$.

\textbf{Estimate density of $X_1$.}
Combining Figure \ref{fig:complexity} and Figure \ref{fig:denest} concludes that iterated MLMC with projected coefficients reduce the computational cost without losing the accuracy in terms of density estimation.
\begin{figure}
\centering
\begin{subfigure}{.49\textwidth}
  \centering
  \includegraphics[width=.9\linewidth]{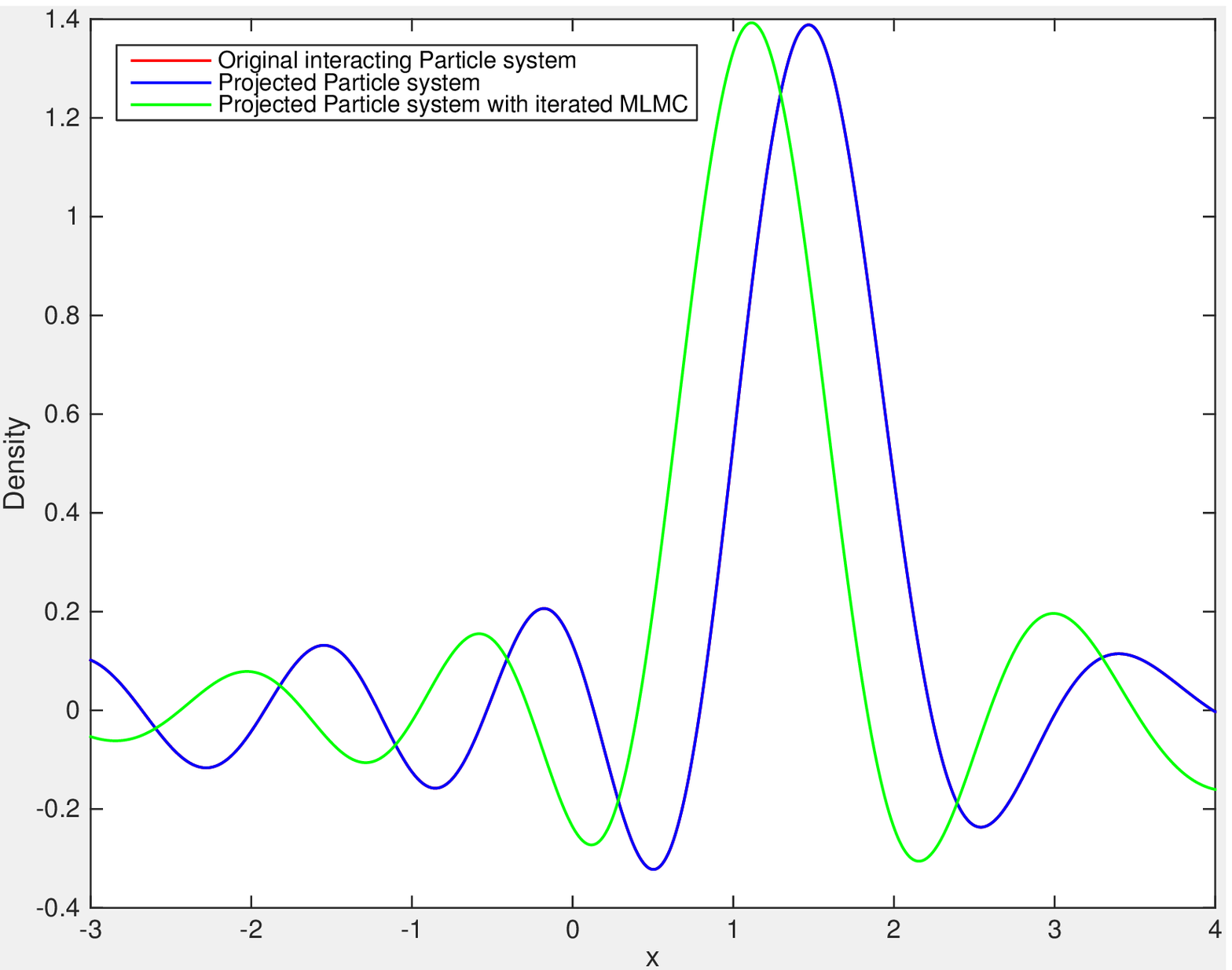}
  \caption{$1$th Picard step}
  \label{fig:noninteractingexample11a}
\end{subfigure}
\begin{subfigure}{.49\textwidth}
  \centering
  \includegraphics[width=.9\linewidth]{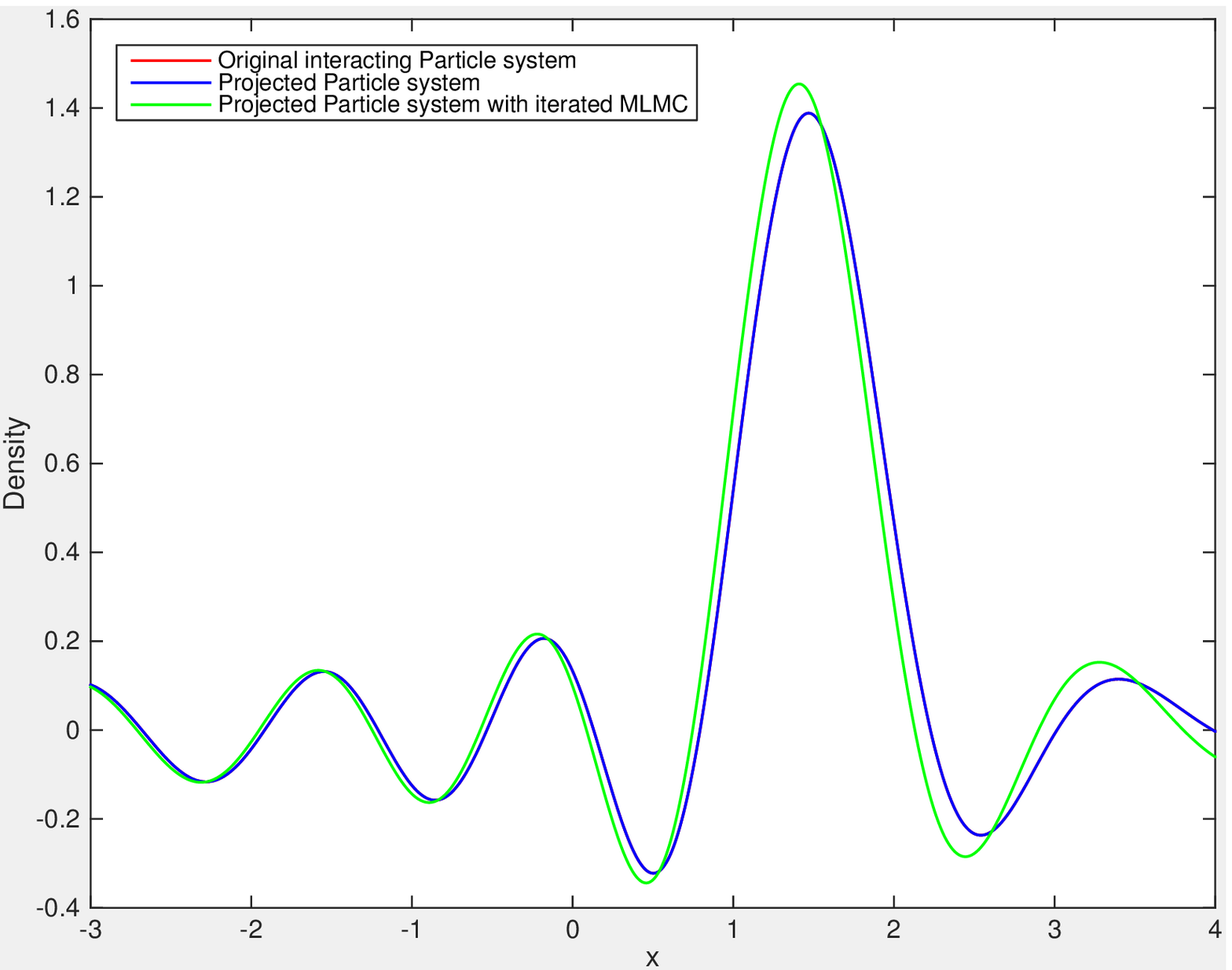}
  \caption{$2$nd Picard step}
  \label{fig:noninteractingexample11b}
\end{subfigure}
\begin{subfigure}{.49\textwidth}
  \centering
  \includegraphics[width=.9\linewidth]{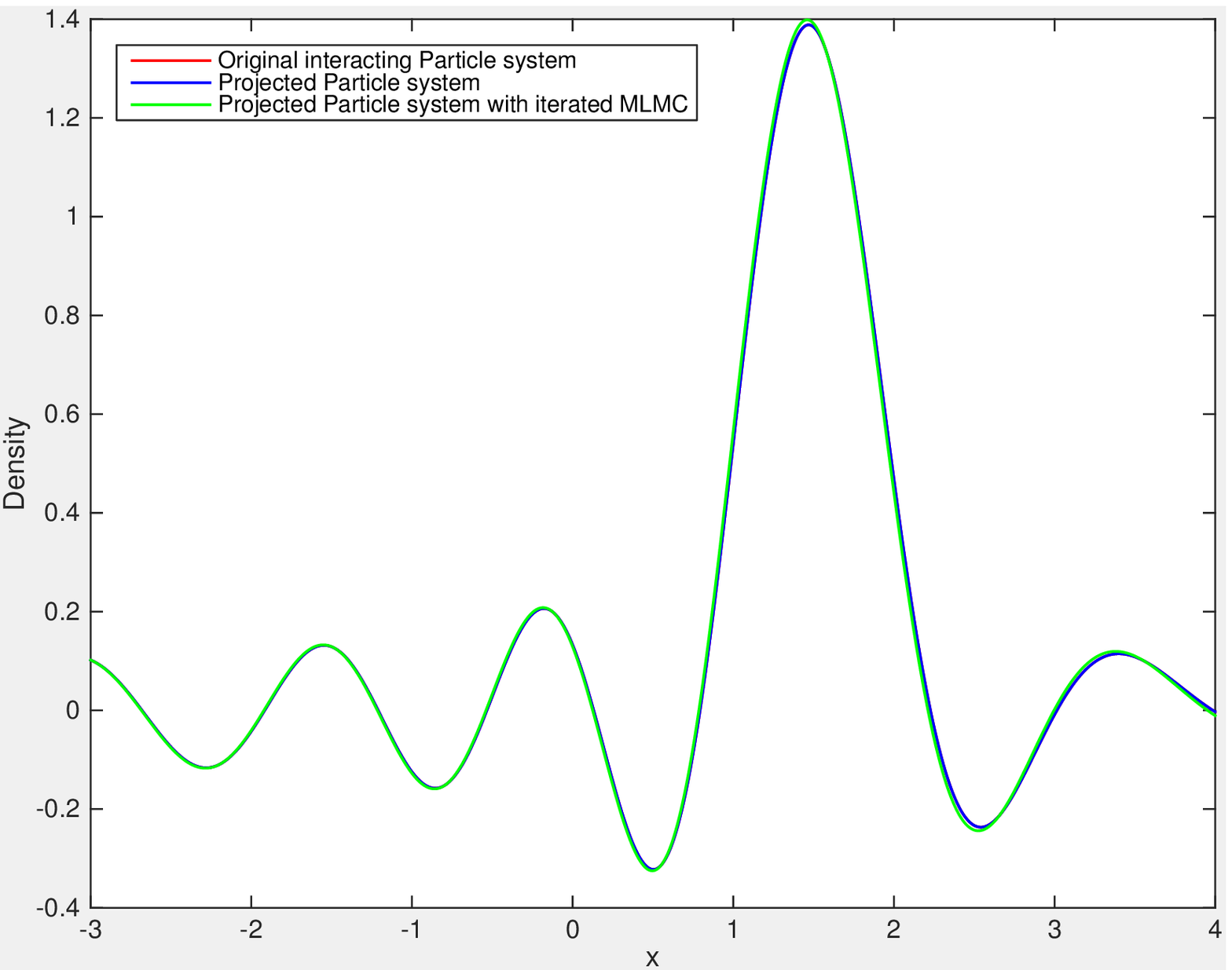}
  \caption{$3$th Picard step}
  \label{fig:noninteractingexample11c}
\end{subfigure}
\centering
\begin{subfigure}{.49\textwidth}
  \centering
  \includegraphics[width=.9\linewidth]{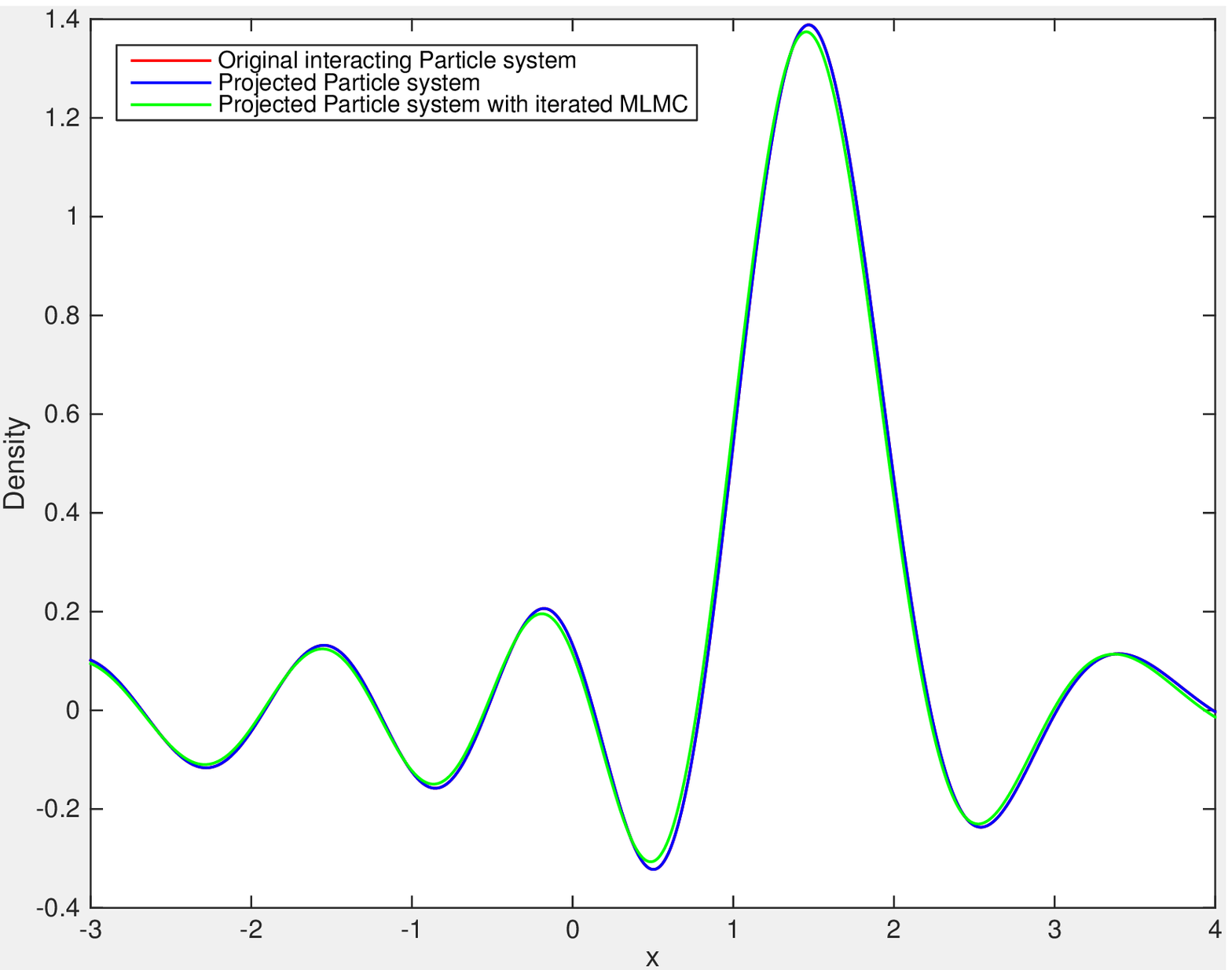}
  \caption{$4$th Picard step}
  \label{fig:noninteractingexample11d}
\end{subfigure}
\caption{Estimated density of $X_1$ using 10 basis functions both for Projected particle system and iterated MLMC with projected coefficients during all Picard steps.}
\label{fig:denest}
\end{figure}

\section{Conclusions}

In this article, we presented a novel iterative MLMC estimator for the challenging problem of simulating McKean-Vlasov SDEs. As in the classical stochastic particle, approximation particles are not independent and the bias and statistical error are in a nonlinear relationship, classical methods fails. Nonetheless in the case when coefficients of MVSDEs are smooth enough our approach recovers computational complexity \(\varepsilon^{-2} |\log(\epsilon)|^{3}\). This is a very promising strand of research and in the future work and we have addressed the complete rigorous analysis of this fact under certain assumptions. Furthermore, our approach easily extends to other than projections approximation methods and also allows for MLMC treatment of approximating densities. Finally, we believe that the idea of approximating a complex/non-linear/highly-dimensional models with manageable/simplified models and then applying iterative MLMC approach can be fruitful in other application areas and we anticipate interest in this approach from various research communities.

\begin{acknowledgement}
This work was supported by The Alan Turing Institute under the EPSRC grant EP/N510129/1
\end{acknowledgement}

\bibliographystyle{spmpsci}
\bibliography{Particles.bib}

\end{document}